\documentclass[a4paper,11pt]{amsproc}
\usepackage{amsmath,amsthm,amssymb,amscd}
\usepackage{mathtext}
\usepackage{euscript}
\usepackage[normalem]{ulem}
\usepackage{color}
\usepackage{mathtools}
\usepackage{multirow}

\usepackage{hyperref}

\usepackage[matrix,arrow,curve]{xy}
\usepackage{graphicx}
\usepackage{bbold}
\input{xypic}
\DeclareGraphicsExtensions{.jpg}

\newtheorem{theorem}{Theorem}[section]
\newtheorem{proposition}[theorem]{Proposition}
\newtheorem{corollary}[theorem]{Corollary}

\newtheorem{question}[theorem]{Question}
\newtheorem*{theorem*}{Theorem}
\theoremstyle{definition}
\newtheorem{definition}[theorem]{Definition}
\newtheorem{example}[theorem]{Example}
\newtheorem{construction}[theorem]{Construction}
\theoremstyle{remark}
\newtheorem{remark}[theorem]{Remark}

\numberwithin{equation}{section}

\newcommand{\Hom}{\mathop{\mathrm{Hom}}\nolimits}

\def\C{\mathbb C}
\def\Q{\mathbb Q}

\def\Z{\mathbb Z}
\newcommand{\zk}{\mathcal{Z}_{\mathcal{K}}}
\newcommand{\cpk}{\left(\mathbb{C} P^{\infty}\right)^{\mathcal{K}}}

\newcommand{\K}{\mathcal{K}}
\newcommand{\cp}{\mathbb{C} P^{\infty}}
\newcommand{\sk}{(\underline{S})^{\K}}
\newcommand{\stwok}{(S^2)^{\K}} 
\newcommand{\ah}{\mathbf{AH}}

\newcommand{\cux}{CU_{*}(\Omega X)}

\def\ge{\geqslant}
\def\geq{\geqslant}

\def\leq{\leqslant}

\def\Bar{\mathop{\mathrm{Bar}}}
\def\cobar{\mathop{\mathrm{Cobar}}}
\def\coker{\mathop{\mathrm{Coker}}}
\def\sgn{\mathop{\mathrm{sgn}}}
\def\colim{\mathop{\mathrm{colim}}}

\begin{document}

\title[]{
Adams--Hilton models and higher Whitehead brackets for polyhedral products
%Модели Адамса--Хилтона и произведения Уайтхеда некоторых полиэдральных произведений
}

\author{Elizaveta Zhuravleva}
\address{Lomonosov Moscow State University, Russian Federation;\newline
National Research University Higher School of Economics, Russian Federation}
\email{ahertip@gmail.com}

\thanks{
\newline
This work is supported by Russian Science Foundation under grant № 21-71-00049.
}

\begin{abstract}
%В работе строятся модели Адамса--Хилтона для полиэдральных произведений сфер и %пространств Дэвиcа--Янушкевича. Мы показываем, что в этих случаях  модель
%Адамса--Хилтона можно выбрать так, чтобы она совпадала с кобар-конструкцией от %коалгебры гомологий. В качестве приложения мы предъявляем новый способ работы с %итерированными высшими произведениями Уайтхеда --- построение цепей 
%в кобар--конструкции, представляющих классы образа отображения Гуревича.

In this paper, we construct Adams--Hilton models for the polyhedral products of spheres $\sk$ and Davis--Januszkiewicz spaces $\cpk$. We show that in these cases the
Adams--Hilton model can be chosen so that it coincides with the cobar construction of the homology coalgebra. 
We apply the resulting models to the study of iterated higher Whitehead products in $\cpk$.
Namely, we explicitly construct a chain
in the cobar construction representing the homology class of the Hurewicz image of a Whitehead product.
\end{abstract}

\maketitle

\tableofcontents

%\section*{Введение}
\section*{Introduction}

For each simplicial complex $\K$ the following two topological spaces are defined, which are  the main objects in toric topology: the
moment--angle complex $ \zk $ and the Davis--Januszkiewicz space $ \cpk $.
These spaces are particular examples of a more general construction, the polyhedral
product
$(\underline{X}, \underline{A})^{\K}$; see \cite{tortop} for the details 
of this construction.

%Каждому симплициальному комплексу $\K$ на множестве вершин $[m]$ мы можем сопоставить следующие топологические пространства, являющимися главными объектами в торической топологии: момент--угол комплекс $\zk$ и пространство Дэвиса--Янушкевича $\cpk$. 
%Данные пространства являются частными случаями более общего объекта --- полиэдрального
%произведения, которое также определеляется по некоторому симплициальному комплексу $\K$ [123123].
%Имеется большое количество работ, сконцентрированных
%на изучении гомотопического типа данных пространств, гомологических и когомологических %свойств и инвариантов [123123]. В данной работе мы изучаем алгебру Понтрягина
%$H_{*}(\Omega \cpk)$
%пространства 
%Дэвиса--Янушкевича $\cpk$. Алгебра Понтрягина $H_{*}(\Omega\zk)$ момент--угол комплекса можно
%отождествить с коммутаторной подалгеброй в $H_{*}(\Omega \cpk)$ [123123].

The homotopy theory  of polyhedral products is an active research area,
see the survey~\cite{BBCG2}. 
In this paper, we look into the Pontryagin algebra
$ H_{*} (\Omega \cpk) $ of the
Davis--Januszkiewicz  space $ \cpk $. The Pontryagin algebra $ H_{*}(\Omega \zk) $ of the moment--angle complex can be
identified with the commutator subalgebra of $ H_{*}(\Omega \cpk) $.

In \cite{PR}, the Pontryagin algebra $H_{*}(\Omega \cpk)$ was studied from the categorical point of view. In particular, the Pontryagin algebra was described completely there  by generators and relations in the case when $\K$ is a flag complex. A set of multiplicative generators for $H_{*}(\Omega\zk)$ was presented in~\cite{GPTW}.
When $\K$ is not flag, nontrivial higher Whitehead products appear in the
Pontryagin algebra $H_{*}(\Omega\cpk)$ and the problem of describing its generators and relations becomes difficult. 

%В работе \cite{PR} алгебра Понтрягина $H_{*}(\Omega \cpk)$ изучается с категорной точки зрения. В частности, в этой работе получено полное описание алгебры Понтрягина
%в случае флагового симплициалього комплекса $\K$. В работе \cite{GPTW} описан список мультипликативных образующих в $H_{*}(\Omega\zk)$ в случае флагового симплициального комплекса. В нефлаговом случае изучение $H_{*}(\Omega \cpk)$ усложняется наличием нетривиальных высших итерированных произведений Уайтхеда.
Higher Whitehead products are
invariants of unstable homotopy type, they have been
studied since the 1960s, see \cite{hardie}, \cite{porter}, \cite{wil}. 
Nowadays, they 
play an important role in the study of polyhedral products. 
%Важную роль в изучении алгебр Понтрягина полиэдральных произведений
%играет высшее произведение Уайтхеда, инвариант нестабильной теории гомотопий \cite{hardie}.
A large class of simplicial complexes is formed by those  $ \K $ for which $ \zk $ is a wedge of spheres, and each sphere is mapped as a higher iterated Whitehead product to the Davis--Januszkiewicz space $ \cpk $ \cite{PR, AbPa, Ab}. 
%However, in general
%it is not true that if the moment--angle complex $ \zk $ is homotopy equivalent
%a wedge of spheres, then each sphere is represented by some Whitehead product \cite{Ab}.

%Большой подкласс симплициальных комплексов образуют такие $\K$, для которых $\zk$ есть букет сфер, и каждая сфера отображается как высшее итерированное произведение Уайтхеда двумерных классов в пространство Дэвиса--Янушкевича $\cpk$ \cite{PR, AbPa}. Однако, вообще говоря,
%неверно, что если момент--угол комплекс $\zk$ гомотопически эквивалентен
%букету сфер, тогда каждая сфера представляет собой некоторое произведение Уайтхеда \cite{Ab}.

Adams' cobar construction~\cite{Adams} on the singular chains of $X$ gives a model for the Pontryagin algebra $H_*(\Omega X)$; however, it is not easy to work with. We show that for some polyhedral products one may work directly with the cobar construction of the homology coalgebra. The resulting algebraic models can be calculated explicitly.
This constitutes our first result.

%Для изучения алгебры Понтрягина $H_{*}(\Omega \cpk)$ мы используем
%функтор кобар--конструкции $\Omega$, который сопоставляет сингулярным цепям топологического $X$ сингулярные цепи пространства петель $\Omega X$.
%В случае некоторых полиэдральных произведений можно непосредственно работать
%с кобар--конструкцией от гомологий.

\begin{theorem*} [Prop.~\ref{prop_cobar_sph}] 
Let $\sk =
\bigl( (S^{n_1}, pt), \ldots,  (S^{n_k}, pt)
\bigr)^{\K}$ be a polyhedral product of spheres, $n_i \geq 2$ for each $i$.
We regard its homology $H_{*}(\sk)$ with integer coefficients as a dg coalgebra with
zero differential. Then its cobar construction is given by
$$\cobar H_{*}(\sk)
\cong (T(U), \partial), \,\,
U = \langle b_J, \, J \in \K, \, J \neq \varnothing\rangle,
$$
$$\partial(b_{j_1 \cdots j_s}) = 
\sum_{p=1}^{s-1} \sum_{\theta \in \widetilde{S}(p, s-p)} \varepsilon(\theta)
(-1)^{|b_{j_{\theta(1)} \cdots j_{\theta(p)}}|+1 }
\bigl[b_{j_{\theta(1)} \cdots j_{\theta(p)}}, b_{j_{\theta(p+1)} \cdots j_{\theta(s)}}\bigr].
$$
The homology of $\cobar H_{*}(\sk)$ is isomorphic to the Pontryagin
algebra
$H_{*}(\Omega (\sk))$.
\end{theorem*}

A similar statement holds for the Davis--Januszkiewicz space $\cpk$ \cite{tortop}.

For a simply connected CW-complex $X$, Adams and Hilton~\cite{AH} constructed
a dga $\mathbf{AH}(X)$ over $\Z$. 
Its underlying algebra $T(V)$ is free on a set of generators $V$
corresponding 
to the cells of~$X$. 
The homology of the Adams--Hilton model is the Pontryagin algebra $H_*(\Omega X)$.
Adams--Hilton models have already arisen in the study of polyhedral products~\cite{GT,GIPS}.
There is an indeterminacy 
in the construction of $\mathbf{AH}(X)$, so there is a family of Adams--Hilton models for a given~$X$. However, in the case of our polyhedral products, we prove that there exists a canonical Adams--Hilton model, which  coincides with the cobar construction of the homology coalgebra. 
In general, it is not true that an Adams--Hilton
model is isomorphic to the cobar construction of some coalgebra.

\begin{theorem*}[Th.~\ref{th_AH_cobar_coincide}] 
Let $\K$ be an arbitrary simplicial complex.
Then there exist Adams--Hilton models $\ah(\sk)$ and $\ah(\cpk)$, such that
$$
  \ah(\sk)  \cong \cobar H_{*}(\sk),\qquad
  \ah(\cpk)  \cong \cobar H_{*}(\cpk)
$$
in the category DGA.
\end{theorem*}

Adams--Hilton models behave nicely with respect to CW-maps.
If we fix models $\mathbf{AH}(X)$ and $\mathbf{AH}(Y)$ and a CW-map $f\colon X \rightarrow Y$,
there is a map $AH(f)\colon \mathbf{AH}(X) \rightarrow \mathbf{AH}(Y)$, which
in some cases can be calculated explicitly. 
This turns out to be useful when we work with higher Whitehead products.

In particular, we have the following result about the iterated higher Whitehead product considered in \cite{AbPa}.
%Преимущество использования моделей Адамса--Хилтона заключается в том, что 
%эти модели имеют хорошие свойства по отношению к отображениям.
%Если зафиксировать модели $\mathbf{AH}(X)$ и $\mathbf{AH}(Y)$, каждому отображению
%$CW$-комплексов 
%$f\colon X \rightarrow Y$ можно сопоставить отображение моделей 
%$AH(f)\colon \mathbf{AH}(X) \rightarrow \mathbf{AH}(Y)$, которое в некоторых случаях можно полностью описать. Это оказывается полезным в случае (канонического) итерированного высшего произведения Уайтхеда.

\begin{theorem*}[Th.~\ref{th_wh}] Let $[[\mu_1, \mu_2, \mu_3], \mu_4, \mu_5] \in \pi_7(\Omega \cpk)$
be the canonical iterated higher Whitehead product in $\cpk$, given by the composite
$$
S^8 \longrightarrow T(S^5, S^2, S^2) \longrightarrow (S^2)^{\K}
\longrightarrow \cpk.
$$
Then Adams--Hilton models can be constructed explicitly for each map in the sequence above:
$$
\mathbf{AH}(S^8) \longrightarrow \mathbf{AH}(T(S^5, S^2, S^2)) \longrightarrow 
\mathbf{AH}((S^2)^{\K}) \longrightarrow \mathbf{AH}(\cpk).
$$
Furthermore, the Hurewicz image of the higher Whitehead product
$[[\mu_1, \mu_2, \mu_3], \mu_4, \mu_5]$ is represented by the following cycle in the cobar construction 
$\cobar H_{*}(\cpk) \cong \mathbf{AH}(\cpk)$:
$$
- \partial \bigl( 
[\chi_{123}, \chi_{45}] + [\chi_{1234}, \chi_5] + [\chi_{1235}, \chi_4] 
\bigr) .
$$
\end{theorem*}

I wish to express gratitude to my supervisor Taras E. Panov for stating the problem, valuable advice and
stimulating discussions. I would also like to thank 
Theoretical Physics and Mathematics Advancement Foundation “BASIS”,
of which I am a scholarship recipient.

\section{Preliminaries}
A \textit{simplicial complex} $\K$ on the set $[m] = 
\{1, 2, \ldots, m\}$ is a collection of subsets $I \subset [m]$ closed under taking any subsets. We assume that $\K$ contains 
$\varnothing$ and all one-element subsets $\{i\}, \,\,i = 1, \ldots, m$.
We refer to $I \in \K$ as a \textit{simplex} of $\K$.
Denote by $\Delta^{m-1}$ or by $\Delta(1, \ldots, m)$ the
\textit{full simplex} on the set 
$[m]$.
It is a simplicial complex consisting of all subsets of $[m]$. 
Similarly, denote by $\Delta(I)$ the full simplex on the set  
$I = \{i_1, \ldots, i_s\}$. 
The condition $I \in \K$ implies 
$\Delta(I) \subset \K$.
The simplicial complex $\partial \Delta(I)$ consists of all simplices  $J, \,\,J \in \Delta(I),$ except the simplex $I$.

\begin{construction}
Let $\K$ be a simplicial complex on the set $[m]$. 
Suppose we are given an ordered set of $m$ pairs of based CW-complexes
$$
(\underline{X}, \underline{A}) =
\{ (X_1, A_1), \ldots, (X_m, A_m)
\}.
$$

For each $I \in \K$ we consider the set 
$$
(\underline{X}, \underline{A})^{I} =
\{
(x_1, \ldots, x_m) \in \prod\limits_{k = 1}^{m} X_k
\mid
x_k \in A_k \,\,\text{if} \,\, k \notin I
\}.
$$

We define the \textit{polyhedral product} of  pairs
$(\underline{X}, \underline{A})$ corresponding to $\K$ as
$(\underline{X}, \underline{A})^{\K}$, where
$$
(\underline{X}, \underline{A})^{\K}
= \bigcup\limits_{I \in \K} (\underline{X}, \underline{A})^{I}
= \bigcup\limits_{I \in \K} \Bigl(
\prod\limits_{i\in \K} X_i 
\times
\prod\limits_{i\notin \K} A_i
\Bigr)
\subset \prod\limits_{k = 1}^{m} X_k.
$$
\end{construction}

If $X = X_i$ and $A = A_i$ for all $i$ we use the notation
$(X, A)^{\K}$ instead of
$(\underline{X}, \underline{A})^{\K}$. 
If $A_i = pt$ for all $i$ we write $(\underline{X})^{\K}$.

\begin{example}
The polyhedral product $(D^2, S^1)^{\K}$ is called the
\textit{moment--angle complex}; it is denoted by $\zk$.
\end{example}

\begin{example}
The polyhedral product $(\cp)^{\K} 
= (\cp, pt)^{\K}$ is homotopy equivalent to the
\textit{Davis--Januszkiewicz space} $DJ(\K)$.
\end{example}

\begin{example}
Denote by
$(\underline{S})^{\K} = 
(\underline{S}, pt)^{\K}$
the polyhedral product of spheres:
$$
(\underline{S}, pt)^{\K} =
\bigl( (S^{n_1}, pt), \ldots,  (S^{n_m}, pt)
\bigr)^{\K},  \quad n_i \geqslant 2\text{ for each} \,\, i.
$$
A special case is the space $\stwok$.
\end{example}

\begin{theorem}[{\cite[Th.~4.3.2]{tortop}}]
The moment-angle complex $\zk$ is the homotopy
fibre of the canonical inclusion
$\cpk \hookrightarrow (\cp)^m$.
\end{theorem}

The fibre inclusion $\zk \hookrightarrow \cpk$ can be described explicitly.
Consider the map of  pairs $(D^2, S^1) \rightarrow (\cp, pt)$ which sends the interior of $D^2$ homeomorphically onto the 2-dimensional  cell in  $\mathbb{C}P^1$, and sends
the boundary of the disk to the basepoint.
Then we have the induced map of the polyhedral products
$\zk = (D^2, S^1)^{\K} \rightarrow (\cp, pt)^{\K}$.

\begin{corollary}
There is an exact sequence of rational homotopy Lie algebras
$$
0 \rightarrow \pi_{*}(\Omega \zk) \otimes \Q \rightarrow
\pi_{*}(\Omega \cpk) \otimes \Q \rightarrow 
\langle u_1, \ldots, u_m \rangle_{\Q} \rightarrow 0
$$
where
$\langle u_1, \ldots, u_m \rangle_{\Q}$
denotes the $m$-dimensional vector space with trivial Lie bracket, 
$ |u_i|=1$; and an exact sequence of Pontryagin algebras
$$
1 \rightarrow H_{*}(\Omega \zk; \textbf{k}) 
\rightarrow
H_{*}(\Omega \cpk; \textbf{k})
\rightarrow \Lambda[u_1, \ldots, u_m] \rightarrow 0
$$
for any commutative ring $\textbf{k}$ with unit.
\end{corollary}

\begin{corollary}
Any map
$S^n \rightarrow \cpk$ with $n \geqslant 3$ lifts to a map $S^n \rightarrow \zk$. 
\end{corollary}

The following combinatorial construction is crucial for the definition of
canonical higher Whitehead brackets for polyhedral products; it can be found in \cite{AbPa}.

\begin{definition} \label{def_substit}
Let $\K$ be a simplicial complex on the set of vertices $[m]$,  and let
$\K_1, \ldots, \K_m$ be an ordered collection of simplicial complexes such that
sets of vertices of
each two $\K_i$ and $\K_j$ do not intersect.
The
\textit{substitution complex} 
$\K(\K_1, \ldots, \K_m)$ is defined as follows:
$$
\K(\K_1, \ldots, \K_m) =\{ I_{j_1} \sqcup \cdots \sqcup I_{j_k}
\mid I_{j_l} \in \K_{j_l}, \,\,l = 1, \ldots, k, \,\,\{j_1, \ldots, j_k\} \in \K
\}.
$$
\end{definition}

It is easy to see that $\K(\K_1, \ldots, \K_m)$ is a simplicial complex.

\begin{example}
Suppose each $\K_i$ is a single vertex $\{i\}$; then 
$\K(\K_1, \ldots, \K_m) = \K.$ 
In particular,
$$
  \partial \Delta(1, \ldots, m) = \partial \Delta^{m-1}, \quad
  \Delta(1, \ldots, m) =  \Delta^{m-1},
$$
in accordance with the previous notation.
%As before, we omit the dimension of the simplex nad use  notations
%$\partial \Delta(1, \ldots, m)$ and $\Delta(1, \ldots, m)$.
\end{example}

\medskip
\medskip

\hspace{+18ex}
\includegraphics[scale=0.4]{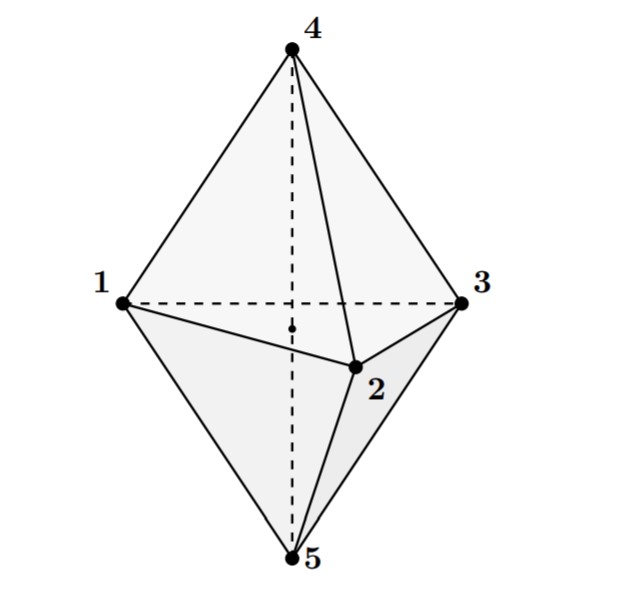}

\begin{center}
Figure 1. Substitution complex $\partial\Delta (\partial\Delta(1, 2, 3), 4, 5)$.
\end{center}

\medskip
\medskip

\begin{example}
Let $\K =\partial\Delta^{2}$, let $\K_1 = \partial\Delta(1, 2, 3),$  let
$\K_2 = \{\varnothing, \{4\} \},$ and let $\K_2 = \{\varnothing, \{5\} \}$. Then the substitution complex has the following form:
$$
\K(\K_1, \K_2, \K_3) = \partial\Delta \bigl(\partial\Delta(1, 2, 3), 4, 5\bigr).
$$
It is shown in Fig. 1.
\end{example}

\section{Canonical higher Whitehead products for the Davis--Januszkiewicz space}

First, we give the definition of a (general) higher Whitehead product, following
\cite{arc}, \cite{hardie}.
Given the spheres
$S^{n_{1}}, \ldots, S^{n_{k}}, n_j \ge 2,$
we consider the wedge 
$S^{n_{1}} \vee \cdots \vee S^{n_{k}}$ and
the
\textit{fat wedge} 
$T\left(S^{n_{1}}, \ldots, S^{n_{k}}\right)$, which is defined by the equation 
$$
S^{n_{1}} \times \cdots \times S^{n_{k}} =
T\left(S^{n_{1}}, \ldots, S^{n_{k}}\right) 
 \cup_\omega e^{N},
$$
where $e^{N}, \,N=n_{1}+\cdots+n_{k},$ is the top cell of the product
 $S^{n_{1}} \times \cdots \times S^{n_{k}}$.
Let
$\omega \colon S^{N-1} \rightarrow T\left(S^{n_{1}}, \ldots, S^{n_{k}}\right)$ be
the attaching map that
corresponds to this cell.

\begin{definition} \label{def_whitehead}
Let $X$ be a topological space. 
Suppose given  homotopy classes
$x_{j} \in \pi_{n_{j}}(X),$ $j=1, \ldots, k.$ 
Denote by $g$ the induced map
$g=x_{1} \vee \ldots \vee x_{k} \colon S^{n_{1}} \vee \cdots \vee S^{n_{k}} \rightarrow X$. 
The 
\textit{(general) higher Whitehead product}
$\left[x_{1}, \ldots, x_{k}\right]_{W} \subset \pi_{N-1}(X)$ is the set of homotopy classes
$$
\{f \circ \omega \mid f \colon T\left(S^{n_{1}}, \ldots, S^{n_{k}}\right) \rightarrow X \text { is an extension of } g\}.
$$
This is described by the diagram
$$
\xymatrix{
&S^{n_{1}} \vee \cdots \vee S^{n_{k}}\ar@{^{(}->}[d]\ar[r]^-g&X\\
S^{N-1}\ar[r]^-\omega 
&T\left(S^{n_{1}}, \ldots, S^{n_{k}}\right)\ar@{-->}[ru]_f
}
$$
We say that the higher Whitehead product
is \textit{trivial} if
$0 \in \left[x_{1}, \ldots, x_{k}\right]_{W}$.
\end{definition}

\begin{remark}
In the case $k=2$,  the fat wedge $T(S^{n_1}, S^{n_2})$ coincides with the wedge 
$S^{n_1} \vee S^{n_2}$ and
we obtain the classical definition of the Whitehead product:
$$
[x_1, x_2]_W\colon
S^{N-1} \xrightarrow{\omega} S^{n_1} \vee S^{n_2} \xrightarrow{g} X.
$$
\end{remark}

From now on we work with the Davis--Januszkiewicz space $\cpk$.
As defined above, a (general) higher Whitehead product is a set of homotopy classes.
However, in the case of $\cpk$ a canonical representative can be chosen in this set.

We define 2-dimensional classes $\mu_i$:
$$
\mu_i\colon S^2 \cong \mathbb{C}P^1 \hookrightarrow \cp
\hookrightarrow (\cp)^{\vee m} \hookrightarrow \cpk.
$$

The map $\cp \hookrightarrow (\cp)^{\vee m}$ is the inclusion of the $i$th 
wedge summand.
The inclusion $ (\cp)^{\vee m} \hookrightarrow \cpk$  is well defined since $\K$ contains all  
$\{i\}, \,\,i = 1, \ldots, m$.

Consider the Whitehead product of $\mu_i$ and $\mu_j$, $i \neq j$:
\begin{multline*}
[\mu_i, \mu_j]_W\colon S^3 \cong \partial(D^2 \times D^2) \cong D^2 \times S^1 \cup S^1 \times D^2 \rightarrow S_i^2 \times pt \cup pt \times S_j^2\\
\cong 
(S^2)^{\partial \Delta (i, j)} \hookrightarrow 
(\cp)^{\partial \Delta (i, j)} \hookrightarrow \cpk.
\end{multline*}
Here the composite $S^3 \rightarrow S^2 \vee S^2 $ is the attaching map of the top cell of
$S^2 \times S^2$. The map  
$S^2 \vee S^2 \rightarrow \cpk$ coincides with $\mu_i \vee \mu_j$.

\begin{definition} \label{def_white_higher}
Let the
indices 
$i_1, \ldots, i_n$
be pairwise different.
Suppose that $\partial \Delta(i_1, \ldots, i_n) \subset \K$.
The
\textit{canonical higher Whitehead product} $[\mu_{i_1}, \ldots, \mu_{i_n}]$
is the element in $\pi_{2n-1}(\cpk)$ defined by
$$
[\mu_{i_1}, \ldots, \mu_{i_n}]\colon S^{2n-1} \cong 
\partial (D_{i_1}^2 \times \cdots \times D_{i_n}^2) \cong 
\bigcup\limits_{k=1}^n (D^2_{i_1} \times \cdots \times S^1_{i_k} \times \cdots \times D^2_{i_n}) \rightarrow
$$
$$ \rightarrow
\bigcup\limits_{k=1}^n (S^2_{i_1} \times \cdots \times pt \times \cdots \times S^2_{i_n})
\cong (S^2)^{\partial \Delta (i_1, \ldots, i_n)} 
\hookrightarrow (\cp)^{\partial \Delta (i_1, \ldots, i_n)} \hookrightarrow \cpk.
$$
\end{definition}

Here the map to the fat wedge $S^{2n-1} \rightarrow 
(S^2)^{\partial \Delta (i_1, \ldots, i_n)}
\cong T(S^2, \ldots, S^2)$ 
coincides with the attaching map of the top cell of the product $(S^2)^n.$
In the case of the Davis--Januszkiewicz space $\cpk$,
there is a canonical map from the fat wedge $T(S^2, \ldots, S^2) \cong (S^2)^{\partial \Delta (i_1, \ldots, i_n)} \hookrightarrow \cpk$ (the indices $i_1, \ldots, i_n$ are pairwise different). 
Since the map
$(S^2)^{\partial \Delta (i_1, \ldots, i_n)} \hookrightarrow \cpk$
extends the map
$$\mu_{i_1} \vee \cdots \vee \mu_{i_n}\colon S^2 \vee \cdots \vee S^2 \hookrightarrow \cpk,$$
the canonical higher Whitehead product
$[\mu_{i_1}, \ldots, \mu_{i_n}]$
is contained in the (general) Whitehead product
$[\mu_{i_1}, \ldots, \mu_{i_n}]_W$, defined as a set in Definition \ref{def_whitehead}.

The canonical higher Whitehead product 
$[\mu_{i_1}, \ldots, \mu_{i_n}]$
is well defined if and only if $\partial\Delta(i_1, \ldots, i_n) \subset \K$.
It is easy to show that this homotopy class equals zero when
$\Delta(i_1, \ldots, i_n) \subset \K$. 

For the general higher Whitehead product
$[\mu_{i_1}, \ldots, \mu_{i_n}]_W$
to be well defined,
it is necessary that 
 $[\mu_{i_1}, \ldots, \widehat{\mu_{i_k}}, \ldots, \mu_{i_n}]_W = 0, \,\,k = 1, \ldots, n$.

\begin{proposition}
The (general) higher Whitehead product $[\mu_{i_1}, \ldots, \mu_{i_n}]_W$ is
\begin{itemize}
\item[(1)] defined in $\pi_{2n-1}(\cpk)$
if and only if $\partial \Delta (i_1, \ldots, i_n) \subset \K$;
\item[(2)]
trivial if and only if $\Delta (i_1, \ldots, i_n) \subset  \K$.
\end{itemize}
\end{proposition}

Further we consider a  bracket sequence with 2-dimensional 
classes
$\mu_i$ such that
this sequence
contains only pairwise different indices, for example: 
$$
\Bigl[\mu_1, \mu_2, [\mu_3, \mu_4, \mu_5], \bigl[\mu_6, [\mu_7, \mu_8, \mu_9],  \mu_{10}\bigr], [\mu_{11}, \mu_{12}] \Bigr].
$$
%By induction, we will define the canonical higher Whitehead product % $w$ corresponding to the bracket sequence. Also we will define 
%the minimal simplicial complex $\partial \Delta_w$,
%for which the homotopy class $w$ is well defined.

\begin{construction} {\cite[Constr. 4.4]{AbPa}} \label{const_siml+compl}
Here we define the substitution simplicial complex $\partial \Delta_w$ corresponding to a bracket sequence $w$
with pairwise different indices.
%By induction, we define the simplicial complex 
%$\partial \Delta_w$,  which corresponds to the bracket sequence $w$.
The bracket $[\mu_{i_1}, \ldots, \mu_{i_n}]$ corresponds to the
complex $\partial\Delta (i_1, \ldots, i_n)$.
%The base of induction is the following.
%e assign the simplicial complex $\partial\Delta (i_1, \ldots, i_n)$
%to the bracket $[\mu_{i_1}, \ldots, \mu_{i_n}]$.
Suppose we have a bracket sequence
 $w$  of the form
$$
[w_1, \ldots, w_q, \mu_{i_1}, \ldots, \mu_{i_p}],
$$ where
$w_1, \ldots, w_q$ are bracket sequences with the corresponding simplicial complexes 
$\partial \Delta_{w_1}, \ldots,
\partial \Delta_{w_q}$.
%Then the simplicial complex 
%$\partial \Delta_{w_i}$ is well defined and corresponds to
%$w_i$. 
Then we assign to the the bracket sequence $w$ the following substitution simplicial complex  (see Definition \ref{def_substit}):
$$
\partial \Delta_w :=
\partial\Delta (\partial \Delta_{w_1}, \ldots,
\partial \Delta_{w_q}, i_1, \ldots, i_p).
$$
We also define
$$
\Delta_w :=
\Delta (\partial \Delta_{w_1}, \ldots,
\partial \Delta_{w_q}, i_1, \ldots, i_p).
$$ 

\end{construction}

\begin{definition} \label{def_iter_whit}
Assume that $\partial \Delta_w \subset \K$.
The
\textit{canonical higher Whitehead product} 
 $w$ that corresponds to the bracket sequence
 $[w_1, \ldots, w_q, \mu_{i_1}, \ldots, \mu_{i_p}]$
 is defined inductively as follows.
The  bracket $[\mu_{i_1}, \ldots, \mu_{i_n}] 
\in \pi_{2n-1}(\cpk)$ 
 is
described in Definition \ref{def_white_higher}.
By induction, 
suppose that canonical higher Whitehead products $w_1, \ldots, w_q,$
$\deg w_i = k_i$,
are already defined. 
Following \cite[Lemma~3.2]{AbPa}, define $[w_1, \ldots, w_q, \mu_{i_1}, \ldots, \mu_{i_p}]$ as the following element of
 $\pi_{*}(\cpk)$:
 %\multiline
$$
[w_1, \ldots, w_q, \mu_{i_1}, \ldots, \mu_{i_p}]:
S^{k_1 + \cdots + k_q + 2p -1} \cong 
\partial (D^{k_1} \times \cdots \times D^{k_q}
\times D^2_{i_1} \times \cdots \times D^2_{i_p} )
$$
$$
\cong 
\Biggl(
\Bigl(\bigcup\limits_{i=1}^q (D^{k_1} \times \cdots \times S^{k_i-1} \times \cdots \times D^{k_q}) \Bigr) \times D^2_{i_1} \times \cdots \times D^2_{i_p}
\Biggr)
 \, \cup \, \quad \quad
$$
$$
\quad \quad \, \cup \, \Biggl(
D^{k_1} \times \cdots \times D^{k_q} \times
\Bigl(
\bigcup\limits_{j=1}^p (D^2_{i_1} \times \cdots \times S^1_{i_j} \times \cdots \times D^2_{i_p})
\Bigr) \Biggr)
$$
$$
\xrightarrow{\alpha}
\Biggl(
\Bigl(\bigcup\limits_{i=1}^q (S^{k_1} \times \cdots \times pt \times \cdots \times S^{k_q}) \Bigr) \times (D^2, S^1)^{ \Delta (i_1, \ldots, i_p)} \Biggr) \, \cup \,
\quad \quad
$$
$$
\quad \quad
\, \cup \, \Biggl( S^{k_1} \times \cdots \times S^{k_q} \times
(D^2, S^1)^{\partial \Delta (i_1, \ldots, i_p)}  \Biggr)
$$
$$
\xrightarrow{\beta}
\Biggl(
\Bigl(\bigcup\limits_{i=1}^q (S^{k_1} \times \cdots \times pt \times \cdots \times S^{k_q}) \Bigr) \times (S^2)^{ \Delta (i_1, \ldots, i_p)} \Biggr) \, \cup \,
\quad \quad
$$
$$
\quad \quad
\, \cup \, \Biggl( S^{k_1} \times \cdots \times S^{k_q} \times
(S^2)^{\partial \Delta (i_1, \ldots, i_p)} \Biggr)
$$
$$
\xrightarrow{\gamma}
\Biggl(
\Bigl(\bigcup\limits_{i=1}^q ((S^{2})^{\partial \Delta_{w_1}} \times \cdots \times pt \times \cdots \times (S^{2})^{\partial \Delta_{w_q}}) \Bigr) \times (S^2)^{ \Delta (i_1, \ldots, i_p)}
\Biggr) \, \cup \,
\quad \quad
$$
$$
\quad \quad
\, \cup \, \Biggl( (S^{2})^{\partial \Delta_{w_1}} \times \cdots \times (S^{2})^{\partial \Delta_{w_q}} \times
(S^2)^{\partial \Delta (i_1, \ldots, i_p)} \Biggr)
$$
$$
\cong
\Biggl(
(S^{2})^{\partial\Delta (\partial \Delta_{w_1}, \ldots, \partial \Delta_{w_q})}  \times (S^2)^{ \Delta (i_1, \ldots, i_p)} \Biggr)
\, \cup \, 
\Biggl((S^{2})^{\Delta (\partial \Delta_{w_1}, \ldots, \partial \Delta_{w_q})}
\times
(S^2)^{\partial \Delta (i_1, \ldots, i_p)} \Biggr)
$$
$$
\cong
(S^{2})^{\partial\Delta (\partial \Delta_{w_1}, \ldots, \partial \Delta_{w_q}, i_1, \ldots, i_p)}
\cong
(S^{2})^{\partial\Delta_w} 
\cong
(\mathbb{C}P^1)^{\partial\Delta_w} 
\xhookrightarrow{\delta} \cpk.
$$

Here the map $\alpha$
is induced by collapsing the boundaries of the disks
$D^{k_i}$ to the point.
The analogous map of the pairs $(D^2, S^1) \rightarrow (S^2, pt)$
induces the map of polyhedral products
$(D^2, S^1)^\K \rightarrow (S^2, pt)^\K$, which we use to define the map  $\beta$. 
The map $\gamma$ is defined  using the maps 
$S^{k_i} \xrightarrow{w_i} (S^2)^{\partial \Delta_{w_i}} 
\hookrightarrow 
(\cp)^{\partial \Delta_{w_i}}$ from the previous step.
Since $\K \supset \partial \Delta_{w}
=
\partial\Delta (\partial \Delta_{w_1}, \ldots,
\partial \Delta_{w_q}, i_1, \ldots, i_p)$, the last map $\delta$ is well defined.
%The latest map  $\delta$ is induced by the inclusion 
%$\partial \Delta_w \subset \K$.
\end{definition}

\begin{proposition}
Suppose that  $\partial \Delta_w \subset \K$. 
Then the canonical higher Whitehead product $w$ is contained
in the corresponding (general) Whitehead product.
\end{proposition}

\begin{proof}
The domain of the map $\gamma$ from Definition \ref{def_iter_whit},
$$
\biggl(
\Bigl(\bigcup\limits_{i=1}^q (S^{k_1} \times \cdots \times pt \times \cdots \times S^{k_q}) \Bigr) \times (S^2)^{ \Delta (i_1, \ldots, i_p)} \biggr) 
 \cup  \biggl( S^{k_1} \times \cdots \times S^{k_q} \times
(S^2)^{\partial \Delta (i_1, \ldots, i_p)} \biggr)
$$
is the fat wedge $T(S^{k_1}, \ldots, S^{k_q}, S^2_{i_1}, \ldots, S^2_{i_p})$.
The composition $\beta \circ \alpha$
$$
S^{k_1 + \cdots + k_q + 2p -1} \xrightarrow{\beta \circ \alpha}
T(S^{k_1}, \ldots, S^{k_q}, S^2_{i_1}, \ldots, S^2_{i_p})
$$
coincides with the attaching map of the top cell of the product
$S^{k_1} \times \cdots \times S^{k_q}
\times S^2_{i_1} \times \cdots \times S^2_{i_p}$.
The restriction of the map $\delta \circ \gamma$
$$
T(S^{k_1}, \ldots, S^{k_q}, S^2_{i_1}, \ldots, S^2_{i_p}) 
\xrightarrow{\gamma} (S^2)^{\partial \Delta_w}
\xhookrightarrow{\delta} \cpk
$$
to the wedge
$S^{k_1} \vee \cdots \vee S^{k_q}
\vee S^2_{i_1} \vee \cdots \vee S^2_{i_p}$
is the wedge of maps
$w_1, \ldots, w_q, \mu_{i_1}, \ldots, \mu_{i_p}$. 
Therefore, the composition
 $w = \delta \circ \gamma \circ \beta \circ \alpha$
\begin{multline*}
w = [w_1, \ldots, w_q, \mu_{i_1}, \ldots, \mu_{i_p}]:
S^{k_1 + \cdots + k_q + 2p -1} 
\\
\xrightarrow{\beta \circ \alpha}
T(S^{k_1}, \ldots, S^{k_q}, S^2_{i_1}, \ldots, S^2_{i_p})
\xrightarrow{\delta \circ \gamma}  \cpk,
\end{multline*}
is indeed an element of the (general) higher Whitehead product
from  Definition~\ref{def_whitehead}.
\end{proof}

\begin{corollary}
Assume that $\partial \Delta_w \subset \K$.
Then the (general) Whitehead product $w$ is nonempty.
\end{corollary}

\begin{question}
Let
$w= [w_1, \ldots, w_q, \mu_{i_1}, \ldots, \mu_{i_p}]_W $ be a (general) Whitehead product. 
Suppose that the product $w$ is defined in $\pi_{*}(\cpk),$ 
in other words, there exists at least one extension
from the wedge to the fat wedge
in  Definition~\ref{def_whitehead}. 
Does it imply that $\partial \Delta_w \subset~\K$?
\end{question}

\begin{theorem}[{\cite[Th.~5.1]{AbPa}}]
Let $\K = \partial \Delta_w$. Then the canonical higher Whitehead product
$w = [w_1, \ldots, w_q, \mu_{i_1}, \ldots, \mu_{i_p}]$ is nontrivial in 
$\pi_{k_1 + \cdots + k_q + 2p -1}(\cpk)$.
\end{theorem}

\begin{construction}
Let
$w$ be a
canonical higher Whitehead product.
Suppose that $\K \supset \Delta_w$.
Then the \textit{canonical extension} of $w$ to the disk is defined:
$$
D^{k_1 + \cdots + k_q + 2p} \cong 
D^{k_1} \times \cdots \times D^{k_q}
\times D^2_{i_1} \times \cdots \times D^2_{i_p} \rightarrow 
S^{k_1} \times \cdots \times S^{k_q}
\times S^2_{i_1} \times \cdots \times S^2_{i_p}
$$
$$
\rightarrow
(S^{2})^{\partial \Delta_{w_1}} \times \cdots \times (S^{2})^{\partial \Delta_{w_q}} \times
S^2_{i_1} \times \cdots \times S^2_{i_p} \cong
(S^{2})^{\partial \Delta_{w_1}} \times \cdots \times (S^{2})^{\partial \Delta_{w_q}} \times
(S^2)^{ \Delta (i_1, \ldots, i_p)}
$$
$$
\cong (S^2)^{\partial \Delta_{w_1} * \cdots * \partial \Delta_{w_q} * \Delta (i_1, \ldots, i_p)}
\cong (S^2)^{\Delta (\partial \Delta_{w_1}, \ldots, \partial \Delta_{w_q}, 
i_1, \ldots, i_p)}
\hookrightarrow \cpk.
$$
This implies that $w = [w_1, \ldots, w_q, \mu_{i_1}, \ldots, \mu_{i_p}] =0.$

\end{construction}

It follows that a canonical higher Whitehead product
$w= [w_1, \ldots, w_q, \mu_{i_1}, \ldots, \mu_{i_p}] $ 
is well defined if and only if $\partial \Delta_w \subset \K$ and equals zero in $\pi_{*}(\cpk)$ if
$\Delta_w \subset \K$, as observed in \cite{AbPa}.

\begin{question}
Let
$w= [w_1, \ldots, w_q, \mu_{i_1}, \ldots, \mu_{i_p}]_W $ be a canonical
higher Whitehead product. Suppose that $w_j \neq 0$ for $j = 1, \ldots, q$ and $w = 0$  in $\pi_{*}(\cpk).$  Does it imply that
$\Delta_w \subset \K$?
\end{question}

The answer is positive for products of depth at most two:

\begin{theorem}[{\cite[Th.~5.2]{AbPa}}]
Let $w_j = [\mu_{j_1}, \ldots, \mu_{j_{p_j}} ],
\, j = 1, \ldots, q,$ be nontrivial canonical higher Whitehead
products. 
Consider the following higher Whitehead product:
$$w= [w_1, \ldots, w_q, \mu_{i_1}, \ldots, \mu_{i_p}]. $$
Then $w = 0$ in $\pi_{*}(\cpk)$ if and only if $\K$ contains
$\Delta_w$
as a subcomplex.
\end{theorem}

\begin{question}
Suppose
that the (general) higher Whitehead product \newline
$w= [w_1, \ldots, w_q, \mu_{i_1}, \ldots, \mu_{i_p}]_W $ is trivial and
$w_j $ are nontrivial for $ j = 1, \ldots, q$.
Does it imply that $\Delta_w \subset \K$?
\end{question}

\section{Cobar construction for polyhedral products}

We consider the cobar construction (the algebraic loop functor) $\cobar$
from the category of 1-connected differential graded coalgebras $DGC_{1,\Z}$ to
the category of differential graded connected algebras
$DGA_{0,\Z}$. 
Its right adjoint  is the bar construction (the algebraic classifying functor)
$\Bar$:
$$\cobar\colon DGC_{1,\Z} \leftrightarrows DGA_{0,\Z} : \Bar.$$

The \emph{cobar functor} $\cobar$ assigns to a dg coalgebra $(C,\partial_C)$ with $C_0=\mathbb Z$ and $C_1=0$ the dg algebra 
$$
\cobar C=(T(s^{-1}\overline C),\partial)
$$
that is the free associative algebra on the desuspension of the module 
$\overline C= \coker(C_0 \cong \Z \hookrightarrow C).$
The differential $\partial$ is given~by 
\begin{equation}\label{dcobar}
  \partial( s^{-1} c) = - s^{-1} \partial_C (c)+\sum_i (-1)^{|x_i|} s^{-1} x_i \otimes s^{-1} y_i
\end{equation}
where $c\in C$ with comultiplication  $\Delta c=c\otimes 1+1\otimes c+\sum_i x_i \otimes y_i$.

\begin{theorem}[{\cite{Adams}}]
For a simply connected pointed space $X$ and a commutative ring $\mathbf{k}$ with unit, there is a natural isomorphism of graded algebras
$$H(\cobar C_{*}(X; \mathbf{k})) \cong H_{*}(\Omega X; \mathbf{k}),$$
where $C_{*}(X; \mathbf{k})$ denotes the suitably reduced singular chain complex of $X$.
\end{theorem}

Now we turn to  polyhedral products. 
Recall that $\underline{X}$ stands for the sequence 
of topological spaces
$(X_1, \ldots, X_m)$. Also,  by $\underline{X}^{\K}$ we denote 
the polyhedral product $\bigl( (X_1, pt), \ldots, (X_m, pt)\bigr)^{\K}$.

\begin{theorem}[{\cite[Th.~8.1.2]{tortop}}]
If each space $X_i$ in $\underline{X} = (X_1, \ldots, X_m)$ is formal, then the
polyhedral product $\underline{X}^{\K}$ is also formal.
\end{theorem}

It follows that 
the Davis–Januszkiewicz space $\cpk$ and 
the polyhedral product of spheres
$\sk =
\bigl( (S^{n_1}, pt), \ldots,  (S^{n_k}, pt)
\bigr)^{\K}
$ are formal. As shown in \cite{NR}, the space
$\cpk$ is also $\Z$-formal. The same holds for $\sk$. 
%These results can be derived from the Adams--Hilton models constructed below. 

The \textit{face algebra} $\Z[\K]$
is  the quotient dg algebra with zero differential
$$
\Z[\K] = \Z[v_1, \ldots, v_m]/ \mathcal{I_K}, \,\,|v_i| = 2,
$$
where $\mathcal{I_K} = (v_I \mid I \notin \K)$
is the Stanley--Reisner ideal generated by those monomials $v_I = 
v_{i_1} \cdots v_{i_s}$ for
which $I = \{ i_1, \ldots, i_s\}$ is not a simplex in $\K$.

The
\textit{face coalgebra} $\Z\langle\K\rangle$ is the graded dual of the face algebra $\Z[\K]$. 
We consider multisets $\sigma$ of $m$ elements, 
$$
\sigma =\{\underbrace{1,\ldots,1}_{k_1},\underbrace{2,\ldots,2}_{k_2},
  \ldots,\underbrace{m,\ldots,m}_{k_m}\}
$$
such that the \textit{support} of $\sigma$
(i.e. the set $I_{\sigma} =
\{i \in [m] \mid k_i \neq 0\}
$) is a simplex in $\K$.
Then $\Z\langle\K\rangle$ has an additive basis consisting of the elements $c_{\sigma}$, each $c_{\sigma}$ is dual to
 the monomial
$v_1^{k_1}v_2^{k_2}\cdots v_m^{k_m}\in\Z[\K]$. The
coproduct is given by
$$
  \Delta c_{\sigma}=
  \sum_{(\tau, \tau'), \,\,\sigma=\tau\,\sqcup\,\tau'}c_{\tau}\otimes c_{\tau'},
$$
where the sum ranges over all ordered partitions of $\sigma$ into
submultisets $\tau$ and $\tau'$.

\begin{proposition} We have
\begin{itemize}
\item[(a)]
the cohomology algebra $H^{*}(\cpk; \Z)$ is isomorphic to
$\Z[\K]$,
\item[(b)]
the homology coalgebra $H_{*}(\cpk; \Z)$ is isomorphic to
$\Z\langle\K\rangle$.
\end{itemize}
\end{proposition}

Consider the cobar construction $\cobar \Z \langle \K \rangle$.
By the definition above, we obtain
\begin{equation} \label{cobar_cpk}
\cobar \Z \langle \K \rangle = 
T( U ), \,\, U =\,\langle\chi_{\sigma} \mid 
I_{\sigma} \in \K, \,\sigma \neq \varnothing \rangle, \,\,
 |\chi_{\sigma}| = 2|\sigma| -1.
\end{equation}
Since the differential on $\Z\langle\K\rangle$ is zero, we get
\begin{equation} \label{diff_cpk}
\partial \chi_{\sigma} = \sum_{(\tau, \tau'), \,\,\sigma=\tau\,\sqcup\,\tau'}\chi_{\tau}\otimes \chi_{\tau'}.
\end{equation}
For example,
$$
\partial \chi_{i i} = \chi_i \otimes \chi_i, \quad
\partial \chi_{ i j} = \chi_i \otimes \chi_j + \chi_j \otimes \chi_i.
$$

\begin{proposition}[{\cite[Prop.~8.4.10]{tortop}}]
\label{1}
There is an isomorphism of graded algebras
$$
H_{*}( \Omega \cpk; \Z) \cong H(\cobar \Z 
\langle \K \rangle).
$$
\end{proposition}

%The proposition above is proved by using the integral formality of $\cpk$. We want to obtain the similar statement
%for the polyhedral product of spheres.
%Namely, we aim to prove that the homology  of $\Omega_{*} H_{*} (\sk; \Z)$
%is isomorphic to the Pontryagin algebra. 
%But in the case of $\sk$ we will obtain that the cobar construction coincides with
%an Adams--Hilton model $\mathbf{AH}(X)$ as dga. 
%By definition, the homology algebra of any Adams--Hilton model is isomorphic to the Pontryagin algebra.

We want to obtain a similar description for the cobar construction and the Pontryagin algebra for the polyhedral product of spheres $\sk$. We  start by  describing  explicitly  the cobar
construction $\cobar H_{*} (\sk; \Z)$.

\begin{theorem} [{\cite{BBCG}}]
Let $\underline{X} = (X_1, \ldots, X_m)$ be a sequence of pointed cell
complexes, and let $\mathbf{k}$ be a ring such that the natural map
$$
H^{*}(X_{j_1}; \mathbf{k}) \otimes \cdots \otimes H^{*}(X_{j_k}; \mathbf{k}) 
\rightarrow H^{*}(X_{j_1} \times \cdots \times X_{j_k}; \mathbf{k})
$$
is an isomorphism for any $\{ j_1, \ldots, j_k\} \subset [m]$. Then there is an isomorphism of algebras
$$
H^{*}(\underline{X}^{\K}; \mathbf{k}) \simeq
(H^{*}(X_{1}; \mathbf{k}) \otimes \cdots \otimes H^{*}(X_{m}; \mathbf{k}) ) / 
\mathcal{I},
$$
where $\mathcal{I}$ is the generalised Stanley–Reisner ideal, generated by elements 
$x_{j_1} \otimes \cdots \otimes x_{j_k}$, $x_{j_i} \in 
\widetilde{H}^{*}(X_{j_i}; \mathbf{k})$ and $\{ j_1, \ldots, j_k\}
\notin \K$.
Furthermore, the inclusion 
$\underline{X}^{\K} \hookrightarrow X_1 \times \cdots \times X_m$
induces the quotient projection in cohomology.
\end{theorem}

\begin{corollary}
Let $\sk =
\bigl( (S^{n_1}, pt), \ldots,  (S^{n_k}, pt)
\bigr)^{\K}$ be a polyhedral product of spheres. Then the cohomology algebra is given by
$$
H^{*}(\sk; \mathbb{Z}) = 
\bigl(\mathbb{Z}[a_1] / (a^2_1) \otimes \ldots \otimes \mathbb{Z}[a_m] / (a^2_m) \bigr) \big/ \bigl( a_{j_1} \otimes \cdots \otimes a_{j_s} \,\,\,
\text{if } \,\, \{ j_1, \ldots, j_s\} 
\notin \K \bigr),
$$
where $|a_i| = n_i$.
\end{corollary}

Given $J = \{j_1, \ldots, j_s\} \subset [m]$, $j_1 < \cdots < j_s$, we denote by $\alpha_J$ the element of $H_{*}(\sk; \Z)$ dual to
$a_{j_1} \otimes \cdots \otimes a_{j_s}\in H^*(\sk; \mathbb{Z}) $.

%We use the notation $\alpha_J, \, J = \{j_1, %\ldots, j_s\} \subset [m], \,
%j_1 < \cdots < j_s,$ for the 
%characteristic function in $Hom( H^{*}(\sk; \Z), %\Z)$ that 
%corresponds to the element 
%$a_{j_1} \otimes \cdots \otimes a_{j_s}$.

\begin{proposition}
Let $\sk =
\bigl( (S^{n_1}, pt), \ldots,  (S^{n_k}, pt)
\bigr)^{\K}$ be a polyhedral product of spheres.
Then the homology coalgebra is 
$$
H_{*}(\sk; \Z) = 
\langle \alpha_J,  \,
J 
\in \K \rangle,\quad J = \{j_1, \ldots, j_s\}, \quad |\alpha_J| = n_{j_1} + \cdots + n_{j_s}.
$$
The coproduct is given by
\begin{multline} \label{coprod1}
\Delta \alpha_{j_1 \cdots j_s} = 
\sum_{p=1}^{s-1} \sum_{\theta \in \widetilde{S}(p, s-p)} \varepsilon(\theta)
\bigl(
\alpha_{j_{\theta(1)} \cdots j_{\theta(p)}} \otimes \alpha_{j_{\theta(p+1)} \cdots j_{\theta(s)}}
+ 
\\
+
(-1)^{|\alpha_{...}|
|\alpha_{...}|}
\alpha_{j_{\theta(p+1)} \cdots j_{\theta(s)}}
\otimes \alpha_{j_{\theta(1)} \cdots j_{\theta(p)}}
\bigr),
\end{multline}
where $\widetilde{S}(p, s-p)$ denotes the set of shuffle permutations such that $\theta(1)= 1$, and
$\varepsilon(\theta)$ is the Koszul
sign of the elements $a_{j_1}, \ldots, a_{j_s}$. 
Equivalently,
\begin{equation} \label{coprod2}
\Delta \alpha_J 
=\sum\limits_{(I, L), \,I \sqcup L = J, \,i_1 = j_1} \varepsilon(I, L) 
\bigl(
\alpha_{I} \otimes \alpha_{L}
+
(-1)^{|a_{L}||a_{I}|}  \alpha_{L} \otimes \alpha_{I}
\bigr)
\end{equation}
\begin{equation}  \label{coprod3}
=\sum\limits_{(I, L), \,I \sqcup L = J} \varepsilon(I, L) 
\alpha_{I} \otimes \alpha_{L}.
\end{equation}
Here $I = \{i_1, \ldots, i_p\}$ and $L= \{l_1, \ldots, l_{s-p}\}$ are subsets of~$[m]$, 
$i_1 < \cdots < i_p$, $l_1 < \cdots < l_{s-p}$, and $\varepsilon(I, L)$ is the Koszul sign:
$$
a_{j_1} \otimes \cdots \otimes a_{j_s} =
\varepsilon(I, L)
a_{i_1} \otimes \cdots a_{i_p} \otimes
a_{l_1} \otimes \cdots a_{l_{s-p}}.
$$
\end{proposition}

\begin{proof}
Since
$H^{*}(\sk; \Z)$ is free, we obtain an isomorphism of modules
$H_{*}(\sk; \Z) \cong
\Hom (H^{*}(\sk); \Z), 
$
which is also an isomorphism of coalgebras.
It remains to prove the coproduct formulae.

We denote the algebra $H^{*}(\sk; \Z)$ by $A$ and the coalgebra   $H_{*}(\sk; \mathbb{Z})$ by   $A^{*}$.
Let $\mu\colon A \otimes A \rightarrow A$ be the product in $A$.
The algebra $A$ is finite-dimensional, thus 
$ (A \otimes A )^{*} \cong A^{*} \otimes A^{*}$.
The coproduct $\Delta\colon A^{*} \rightarrow A^{*} \otimes A^{*}$
is defined by
$\Delta\alpha_J = \alpha_J \circ \mu$. 
%Since the function $\alpha_J$ is characteristic, 
%we substitute elements of the form $a_{I} \otimes %a_{L}$, $I \sqcup L = J$, into the product $\mu$.
We have
$$
\mu(a_{I} \otimes a_{L}) = \varepsilon(I, L) a_J.
$$
%where $\varepsilon(I, L)$ is the Koszul sign of %the elements $a_i$:
%$$
%a_{j_1} \otimes \cdots \otimes a_{j_s} =
%\varepsilon(I, L)
%a_{i_1} \otimes \cdots a_{i_p} \otimes
%a_{l_1} \otimes \cdots a_{l_{s-p}}.
%$$
It follows that $\langle a_{I} \otimes a_{L}
\Delta\alpha_J\rangle =\varepsilon(I, L)$.

We use the sign convention $$\langle a_I\otimes a_L,\alpha_{I} \otimes \alpha_{L}\rangle=1.$$
Decompose  $\Delta \alpha_J$ 
into a sum of basis elements in $A^{*} \otimes A^{*}$:
$$
\Delta \alpha_J = \sum_{(I, L), \,i_1 = j_1}
\varepsilon(I, L)  \alpha_{I} \otimes \alpha_{L}
+
\varepsilon(L, I)  \alpha_{L} \otimes \alpha_{I}
$$
$$
=\sum_{(I, L), \,i_1 = j_1} \varepsilon(I, L) 
\Bigl(
\alpha_{I} \otimes \alpha_{L}
+
(-1)^{|a_{L}||a_{I}|}  \alpha_{L} \otimes \alpha_{I}
\Bigr) 
$$
$$
=\sum\limits_{(I, L), \,I \sqcup L = J} \varepsilon(I, L) 
\alpha_{I} \otimes \alpha_{L}.
$$
This gives \eqref{coprod2} and \eqref{coprod3}, and \eqref{coprod1} follows easily.
\end{proof}

\begin{proposition} \label{prop_cobar_sph}
Let $\sk =
\bigl( (S^{n_1}, pt), \ldots,  (S^{n_k}, pt)
\bigr)^{\K}$ be a polyhedral product of spheres, $n_i \geq 2$.
Then
the cobar construction $\cobar H_{*}(\sk; \Z)$
is isomorphic to the  dg algebra $T(U)$, where
$$
U = \langle b_J, \, J \in \K, \,\,
J \neq \varnothing \rangle, \,\, 
J = \{j_1, \ldots, j_s\}, \,\,
|b_J| = n_{j_1} + \cdots + n_{j_s} - 1,
$$
the differential is given by
\begin{equation} \label{cobardif1}
\partial b_J = 
\sum\limits_{(I,L), I\sqcup L = J}
\varepsilon(I, L) (-1)^{|b_I|+1} b_I \otimes b_L,
\end{equation}
where $\varepsilon(I, L)$ is the Koszul sign of the suspended elements $s (b_i), |s (b_i)| = n_i$:
$$
s b_{j_1} \otimes \cdots \otimes s b_{j_s} =
\varepsilon(I, L)
s b_{i_1} \otimes \cdots s b_{i_p} \otimes
s b_{l_1} \otimes \cdots s b_{l_{s-p}}.
$$
Equivalently, 
\begin{equation} \label{cobardif2}
\partial(b_{j_1 \cdots j_s}) = 
\sum_{p=1}^{s-1} \sum_{\theta \in \widetilde{S}(p, s-p)} \varepsilon(\theta)
(-1)^{|b_{j_{\theta(1)} \cdots j_{\theta(p)}}|+1 }
\bigl[b_{j_{\theta(1)} \cdots j_{\theta(p)}}, b_{j_{\theta(p+1)} \cdots j_{\theta(s)}}\bigr],
\end{equation}
where 
$\varepsilon(\theta)$ is the Koszul
sign of the elements
$s b_{j_1}, \ldots, s b_{j_s}$, and
$\widetilde{S}(p, s-p)$ denotes the set of shuffle permutations such that $\theta(1)= 1$. 
\end{proposition}

\begin{proof}
%If $C$ is a coalgebra, then the undelying module %of the cobar construction is $T(s^{-1} %\overline{C})$.
%we prove the formula for the differential.
%If $$\Delta x = \sum_i x_i \otimes y_i,$$ then
%$$d (s^{-1} x) = \sum_i (-1)^{|x_i|} s^{-1} x_i %\otimes s^{-1} y_i.$$
We denote $b_J:=s^{-1} \alpha_J$. Using formula~\eqref{dcobar} for the differential in the cobar construction together with~\eqref{coprod3} and noting that $\partial_C=0$ in our case, we obtain
$$
\partial(s^{-1} \alpha_J) = 
\sum_{(I, L), J = I \sqcup L}
\varepsilon(I, L) (-1)^{|a_{I}|}
s^{-1} \alpha_{I} \otimes s^{-1} \alpha_{L}
= \sum_{(I, L), J = I \sqcup L}
\varepsilon(I, L) (-1)^{|b_{I}| +1}
b_{I} \otimes b_{L}.
$$
It remains to prove \eqref{cobardif2}:
$$
\partial b_J 
= \sum_{(I, L), J = I \sqcup L, i_1 = j_1}
\varepsilon(I, L) (-1)^{|a_{I}| }
b_{I} \otimes b_{L}
+
\varepsilon(L, I)(-1)^{|a_{L}|}
b_{L} \otimes b_{I}
$$
$$
=\sum_{(I, L), J = I \sqcup L, i_1 = j_1}
\varepsilon(I, L) (-1)^{|a_{I}|}
\Bigl(
b_{I} \otimes b_{L}
+
(-1)^{|a_{L}|+ |a_{I}| + |a_{L}||a_{I}|}
b_{L} \otimes b_{I}
\Bigr)
$$
$$
=\sum_{(I, L), J = I \sqcup L, i_1 = j_1}
\varepsilon(I, L) (-1)^{|b_{I}|+1}
\Bigl(
b_{I} \otimes b_{L}
-
(-1)^{ |b_{L}||b_{I}|}
b_{L} \otimes b_{I}
\Bigr).
$$
\qedhere
\end{proof}

\begin{remark}
When all the spheres in $\sk$ are even-dimensional %(e. g. when $\sk = \stwok$),
we obtain the following formula for differential in $\cobar H_{*}( \sk; \Z)$:
%\begin{equation} \label{diff_stwok}
\[
\partial b_J = 
\sum\limits_{(I,L), I\sqcup L = J}
 b_I \otimes b_L.
\]
%\end{equation}
This formula coincides with \eqref{diff_cpk} when $\sigma=I$ (no multiple elements in~$\sigma$). It follows that the cobar construction of $\stwok$ embeds canonically into the cobar construction of $\cpk$. This will be important when studying higher Whitehead products. 

%\eqref{diff_stwok}, we note,
%that they coincide when $J$ is an simplex (i. e. is not an multiset here).
%It is particularly important since the key map in the definition of canonical %Whitehead products is the inclusion $\stwok \hookrightarrow \cpk$.

\end{remark}

\section{Adams--Hilton models}

The main purpose of this section is to verify that Adams--Hilton models for $\sk$ and $\cpk$ coincide with  cobar constructions of corresponding homology coalgebras.

Given a simply connected CW-complex $X$ with a single $0$-dimensional cell and no $1$-dimensional cells, one can construct an \emph{Adams--Hilton model} 
$\mathbf{AH}(X) = (AH(X), d_X)$, a dg algebra over $\Z$
together with a quasi-isomorphism
$$\theta_X \colon \mathbf{AH}(X) \xrightarrow{\simeq} \cux$$ to normalised
cubical chains of the loop space
\cite{AH}. 
Suppose $X$ is a CW-complex of the form
$$X = pt \cup (\bigcup\limits_{\alpha \in S} e_{\alpha}), \quad |e_\alpha| \geq 2.$$
Then the underlying graded algebra $AH(X)$ is
a free associative algebra,  
$$AH(X) = T(V), \quad V = \langle v_{\alpha} \mid \alpha \in S\rangle, \quad
|v_{\alpha}| = |e_{\alpha}|-1,$$
where each generator
$v_\alpha \in V$ corresponds to a cell $e_\alpha$ in $X$.

The maps $\theta_X$ and $\partial_X$ are defined inductively. Suppose we have already constructed an Adams--Hilton model for the
$n$th skeleton:
$$
\mathbf{AH}(X^n) = (T(V_{\leq n-1}), \partial_{X^n}), \quad
\theta_{X^n} \colon (T(V_{\leq n-1}), \partial_{X^n}) \xrightarrow{\simeq} 
CU_{*}(\Omega X^n).
$$
Define $\partial_{X^{n+1}}(v_{\alpha})$ for a generator $v_{\alpha}$,  $|v_{\alpha}| = n$, corresponding to an $(n+1)$-cell cell $e_{\alpha}$.  Consider the attaching map
$f_{\alpha}\colon S^n \rightarrow X^n$ and the generator 
$\beta \in H_{n-1}(\Omega S^n) \cong \Z$. Then $H (\Omega f_{\alpha}) \beta
\in H_{n-1}( \Omega X^n)$ and, by induction, there exists an 
element $z \in T(V_{\leq n-1}), \,\,|z| = n-1,$ such that 
\begin{equation}
(\theta_{X^n})_{*} [z] = H(\Omega f_{\alpha}) \beta.
\end{equation}
Then we define 
\begin{equation}\label{eq_def_AH_dif}
\partial_{X^{n+1}}(v_{\alpha}) = z.
\end{equation}
For the definition of $\theta_{X^{n+1}}(v_{\alpha})$, see \cite{AH}.

There is an indeterminacy in the construction of
$\partial_X$ and $\theta_X$. We have

\begin{proposition} \label{prop1_AH}
Suppose there is an Adams--Hilton model $(AH(X), \partial_X)$. 
Fix a generator $v_{\alpha} \in AH(X) \cong T(V),$ 
$|v_{\alpha}| = n$, corresponding to an $(n+1)$-cell $e_{\alpha}$.
Suppose $\partial_{X^{n+1}}(v_{\alpha}) = z$.
Then for every $a \in T(V_{\leq n-1})$, $|a| =n$,
there exists an Adams--Hilton
 model $(AH(X), \widetilde{\partial}_X)$ with
$$
\widetilde{\partial}_{X^n} = \partial_{X^n}, \quad
\widetilde{\partial}_{X^{n+1}}(v_{\alpha}) = 
z + \partial_{X^n}(a) \quad\text{and}\quad
\widetilde{\partial}_{X^{n+1}}(v_{\beta}) = 
\partial_{X^{n+1}}(v_{\beta})
$$
for any other generator $v_{\beta}$ of degree~$n$,
$v_{\beta} \neq v_{\alpha}$.
\end{proposition}

\begin{proof}
It follows from the inductive construction of $\partial_X$.
\end{proof}

%\begin{remark} \label{remark_AH}
%There is an indeterminacy in the construction of
%maps $\partial_X$ and $\theta_X$, since, for example, the element %$z$ 
%is not unique. We could 
%define
%$\partial_{X^{n+1}}(v_{\alpha})$ as 
%$$ \partial_{X^{n+1}}(v_{\alpha}) = z + \partial_{X^n} (a), %\,\,|a| = n, \,\, 
%a \in T(V_{\leq n-1}).$$
%Therefore there are many Adams--Hilton models
%for a given CW-complex, corresponding to different choices of %$\partial_{X^{n+1}}(v_{\alpha})$. 
%\end{remark}

%The trick with a choice of $\partial_{X^{n+1}}(v_{\alpha})$ has been %used in \cite[Th.~4.1]{AH} and \cite[Th.~4.2]{AH}. We are going %to use it as well.

Fix Adams--Hilton models $(\mathbf{AH}(X), \theta_X)$
and $(\mathbf{AH}(Y), \theta_Y)$
for $X$ and
$Y$, where $AH(X) \cong T(V)$, $AH(Y) \cong T(W)$.
Then for any map $f\colon X \rightarrow Y$ there is an inductive construction~\cite{AH} of a dga-morphism
$\mathbf{AH}(f)\colon \mathbf{AH}(X) \rightarrow \mathbf{AH}(Y)$ and a chain homotopy $\psi_f\colon \mathbf{AH}(X) \rightarrow CU_{*}(\Omega Y)$ between $CU_{*}(\Omega f)\circ\theta_X$ and $\theta_Y\circ\mathbf{AH}(f)$ in the following diagram:
$$
\xymatrix{
\mathbf{AH}(X) \ar[r]^{\theta_X} \ar[d]^{\mathbf{AH}(f)} & CU_{*}(\Omega X) \ar[d]^{CU_{*}(\Omega f)} \\
\mathbf{AH}(Y) \ar[r]^{\theta_Y} & CU_{*}(\Omega Y)
}
$$
We refer to $(\mathbf{AH}(f), \psi_f)$ as an \textit{Adams--Hilton model for} $f$.

Similarly, there is an indeterminacy in the construction of
$\mathbf{AH}(f)$ and $\psi_f$. We have

\begin{proposition} \label{prop2_AH_f}
Suppose there is an Adams--Hilton model $\ah(f)$.
Fix a generator $v_{\alpha} \in AH(X) \cong T(V),$ 
$|v_{\alpha}| = n$, corresponding to an $(n+1)$-cell $e_{\alpha}$.
Suppose ${\ah(f)}(v_{\alpha}) = x$.
Then for every $b \in T(W_{\leq n-1})$, $|b| =n+1$,
there exists an Adams--Hilton model $\widetilde{\ah(f)}$ with
$$
\widetilde{\ah(f)}|_{T(V_{\leq n-1})} = 
{\ah(f)}|_{T(V_{\leq n-1})},
\quad
\widetilde{\ah(f)}(v_{\alpha}) = 
x + \partial_{Y^n}(b)$$
and
$
\widetilde{\ah(f)}(v_{\beta}) = 
{\ah(f)}(v_{\beta})$
for any other generator $v_{\beta}$ of degree~$n$,
$v_{\beta} \neq v_{\alpha}$.
\end{proposition}

\begin{proof}
It follows from the inductive construction of an Adams--Hilton map.
Following the notation of \cite[Th.~3.1]{AH}, one needs to construct a map $\varphi$ inductively on the generators~$a$. By induction,
$\varphi d a=d g_2$ for some~$g_2$, and $\varphi$ is then defined by $\varphi(a) = g_2 + z_2$ for certain~$z_2$. We modify these choices by defining
$\widetilde{g_2} = g_2 + d b$ and $\widetilde{\varphi}(a) = g_2 + d b + z_2$, which again fits the inductive procedure.
\end{proof}

%\begin{remark} \label{remark_AH_f}
%Suppose
%$\ah(f)(v_{\alpha}) = x$, instead we could define %$\ah(f)(v_{\alpha})$ as
%$$ \ah(f)(v_{\alpha}) = x + \partial_{X^n} (a), \,\,|a| = n, \,\, 
%a \in T(V_{\leq n-1}).$$
%By construction in \cite{AH}, we can always extend this choice of %$\ah(f)(v_{\alpha})$ and obtain a model for $f$. 
%\end{remark}

\begin{theorem}[\cite{AH}, \cite{Anick}] \label{th_anick}
Adams--Hilton models have the following properties:

\begin{enumerate}
\item \label{itm:1}
If $f \simeq g$, then $ \mathbf{AH}(f) \simeq \mathbf{AH}(g)$ in DGA. In particular, two models for the
same map must be homotopic in DGA.

\item \label{itm:2}
A model for the identity map $1_X$ is $(1_{\mathbf{AH}(X)}, 0 )$.

\item \label{itm:3}
If $X \xrightarrow{f} Y \xrightarrow{g} Z$ then $\mathbf{AH}(g \circ f)$
may be taken as $\mathbf{AH}(g) \circ \mathbf{AH}(f)$.

\item \label{itm:4}
Let $X_{0}$ be a subcomplex of $X$, say 
$X= pt \cup \left(\bigcup_{\alpha \in S} e_{\alpha}\right)$, 
$X_{0}= pt \cup \left(\bigcup_{\alpha \in S_{0}} e_{\alpha}\right)$, 
$S_{0} \subseteq S$. 
Given any model $\left(AH\left(X_{0}\right), \partial_{X_{0}}, \theta_{X_{0}}\right)$, there is a model
$\left(AH(X), \partial_{X}, \theta_{X}\right)$ 
for which $\partial_{X}$ and $\theta_{X}$ are extensions over $AH(X)$ of $\partial_{X_{0}}$ and $\theta_{X_{0}}$.

\item \label{itm:5}
Under the hypotheses of \eqref{itm:4},
a model for the injection $X_0 \hookrightarrow X$ is
the injection $\ah(X_0) \hookrightarrow \ah(X)$.

\item \label{itm:6}
Under the hypotheses of \eqref{itm:4}, 
let $f\colon X \rightarrow Y$ 
be a map and put 
$f_{0}=\left.f\right|_{X_{0}}$. 
Given any models 
$\left(\mathbf{AH}(Y), \theta_{Y}\right)$ 
and 
$\left(\mathbf{AH}\left(f_{0}\right), \psi_{f_{0}}\right)$, 
there is a model $\left(\mathbf{AH}(f), \psi_{f}\right)$ 
for which $\mathbf{AH}(f)$ and $\psi_{f}$ are extensions over $AH(X)$ of $\mathbf{AH}\left(f_{0}\right)$ and $\psi_{f_{0}}$.

\item \label{itm:7}
Let $\left\{X_{\beta}\right\}$ be a family of subcomplexes of a CW-complex $X$, and suppose $X=\bigcup_{\beta} X_{\beta} .$ Suppose we have models $\left(\mathbf{AH}\left(X_{\beta}\right), \theta_{X_{\beta}}\right)$ satisfying the coherency conditions
$$
\begin{gathered}
\left.\partial_{X_{\beta}}\right|_{AH\left(X_{\beta} \cap X_{\gamma}\right)}=\left.\partial_{X_{\gamma}}\right|_{AH\left(X_{\beta} \cap X_{\gamma}\right)}, \\
\left.\theta_{X_{\beta}}\right|_{AH\left(X_{\beta} \cap X_{\gamma}\right)}=\left.\theta_{X_{\gamma}}\right|_{AH\left(X_{\beta} \cap X_{\gamma}\right)}
\end{gathered}
$$
for each pair of indices $(\beta, \gamma) .$ Then $\operatorname{colim}\left\{\mathbf{AH}\left(X_{\beta}\right), \theta_{X_{\beta}}\right\}$ is an Adams-Hilton model for $X .$

\item \label{itm:8}
Under the hypotheses of \eqref{itm:7}, let $f\colon X \rightarrow Y$ be a map and put $f_{\beta}=\left.f\right|_{X_{\beta}}$. Fixing a model $\left(\mathbf{AH}(Y), \theta_{Y}\right)$, suppose we have models $\left(\mathbf{AH}\left(f_{\beta}\right), \psi_{f_{\beta}}\right)$ satisfying the coherency conditions
$$
\begin{aligned}
\left.\mathbf{AH}\left(f_{\beta}\right)\right|_{AH\left(X_{\beta} \cap X_{\gamma}\right)} &=\left.\mathbf{AH}\left(f_{\gamma}\right)\right|_{AH\left(X_{\beta} \cap X_{\gamma}\right)}, \\
\left.\psi_{f_{\beta}}\right|_{AH\left(X_{\beta} \cap X_{\gamma}\right)} &=\left.\psi_{f_{\gamma}}\right|_{AH\left(X_{\beta} \cap X_{\gamma}\right)} \cdot
\end{aligned}
$$
Then $\operatorname{colim}\left\{\mathbf{AH}\left(f_{\beta}\right), \psi_{f_{\beta}}\right\}$ is an Adams--Hilton model for $f$.

\item \label{itm:9}
Let $f_{0}\colon S^{n} \rightarrow X_{0}$, 
$n \geq 2$, 
and extend $f_{0}$ to 
$f\colon D^{n+1} \rightarrow X=X_{0} \cup_{f_{0}} e^{n+1}$.
Choosing the standard three-cell decomposition of $D^{n+1}$ we have $AH\left(D^{n+1}\right)=$ $T( z, z_{0})$ with $\left|z_{0}\right|=n-1$, $|z|=n$, $\partial(z)=-z_{0}$.
Let $\left(\mathbf{AH}\left(X_{0}\right), \theta_{X_{0}}\right)$ and
$\left(\mathbf{AH}\left(f_{0}\right), \psi_{f_{0}}\right)$ be models for $X_{0}$ and $f_{0}$. 
Then one for $X$ is given by 
$AH(X)=AH\left(X_{0}\right) \otimes T(v_{f})$, 
$\left|v_{f}\right|=n$, 
$\partial_{X}(x)=\partial_{X_{0}}(x)$ 
for 
$x \in AH\left(X_{0}\right)$, 
$\partial_{X}\left(v_{f}\right) =
-\mathbf{AH}\left(f_{0}\right)\left(z_{0}\right)$,
$\theta_{X}(x)=\theta_{X_{0}}(x)$ for $x \in AH\left(X_{0}\right), \theta_{X}\left(v_{f}\right)={CU}_{*}(\Omega f)\left(\theta_{D^{n+1}}(z)\right)+$
$\psi_{f_{0}}\left(z_{0}\right)$.

%\item[(9)]
%Under the hypotheses of $(\mathrm{k})$, let $g_{0}\colon X_{0} \rightarrow Y_{0}$ have model $\left(\mathbf{A}\left(g_{0}\right), \psi_{g_{0}}\right)$. Extend $g_{0}$ to $g\colon X \rightarrow Y=Y_{0} \cup_{g_{0} f_{0}} e^{n+1}$ in the obvious way, and choose models for $X$ and for $Y$ as in (k). Then an extension $\left(\mathbf{A}(g), \psi_{g}\right)$ of $\left(\mathbf{A}\left(g_{0}\right), \psi_{g_{0}}\right)$ may be chosen for which $\mathbf{A}(g)\left(b_{f}\right)=\left(b_{g f}\right), \psi_{g}\left(b_{f}\right)=0$
\end{enumerate}

Further properties concern the model for a product space 
$X \times Y$. 
Suppose
$X=pt \cup\left(\bigcup_{\alpha \in S} e_{\alpha}\right)$ 
and 
$Y= pt \cup\left(\bigcup_{\alpha \in S^{\prime}} e_{\alpha}\right)$, 
then $AH(X \times Y)$ is 
the tensor algebra with the set of generators
$\left\{v_{\alpha} \mid \alpha \in S^{\prime \prime}\right\}$, where
$$
S^{\prime \prime}=S \cup S^{\prime} \cup\left(S \times S^{\prime}\right)
$$
(Adams and Hilton point out that $X \times Y$ need not be a CW complex for their construction to exist.) 
We define the ring homomorphism
$$
\nu_{X Y}\colon AH(X \times Y) \rightarrow AH(X) \otimes AH(Y)
$$
by $\nu_{X Y}\left(v_{\alpha}\right) = v_{\alpha} \otimes 1$ 
for $\alpha \in S$, 
$\nu_{X Y}\left(v_{\alpha}\right)=1 \otimes v_{\alpha}$ 
for $\alpha \in S^{\prime}$, 
$\nu_{X Y} \left(v_{\alpha}\right)=0$ for
$\alpha \in\left(S \times S^{\prime}\right)$.

\begin{enumerate}
\item[(10)]  %\label{itm:10}
Given $\left(\mathbf{AH}(X), \theta_{X}\right)$ and $\left(\mathbf{AH}(Y), \theta_{Y}\right)$, it is possible to choose inductively $\partial_{X \times Y}$ and $\theta_{X \times Y}$ for $AH(X \times Y)$ such that $\nu_{X Y}$ is a dga-morphism and the diagram 

$$
\xymatrix{
\mathbf{AH}(X \times Y)\ar[d]^{\theta_{X \times Y}}  \ar[rr]^{\nu_{X Y}} &  & 
\mathbf{AH}(X) \otimes \mathbf{AH}(Y) \ar[d]^{\theta_X \otimes \,\theta_Y} \\
CU_{*}(\Omega (X \times Y)) \ar[r] & CU_{*}(\Omega X \times \Omega Y) \ar[r] & CU_{*}(\Omega X) \otimes  CU_{*}(\Omega Y)
}
$$
commutes up to chain homotopy and leads to a commutative diagram of isomorphisms of homology rings.

Furthermore, letting $X \stackrel{p_X}{\leftarrow} X \times Y \stackrel{p_Y}{\rightarrow} Y$ denote the projections, we may take $\mathbf{AH}\left(p_{X}\right)=\pi_{1} \nu_{X Y}$ and $\mathbf{AH}\left(p_{Y}\right)=\pi_{2} \nu_{X Y}$.

\item[(11)] % \label{itm:11}
Let $AH(f)\colon AH(X_1) \rightarrow AH(Y_1)$ and
$AH(g)\colon AH(X_2) \rightarrow AH(Y_2)$ be associated with maps
$f\colon X_{1} \rightarrow Y_1$ and $g\colon X_{2} \rightarrow Y_2$,
and let $\partial$, $\theta$ be chosen on $X_1 \times X_2$, 
$Y_1 \times Y_2$ 
so that $\nu_{X_1 X_2}$, $\nu_{Y_1 Y_2}$ 
are dga-morphisms, according to $(10)$. 
Then
we may choose a map
$$
\varphi\colon \ah (X_1 \times X_2) \rightarrow \ah (Y_1 \times Y_2)
$$
so that the following diagram is commutative
$$
\xymatrix{
\mathbf{AH}(X_1 \times X_2) \ar[rr]^{\nu_{X_1 X_2}} 
\ar[d]^{\varphi} && \mathbf{AH}(X_1) \otimes \mathbf{AH}(X_2) 
\ar[d]^{\mathbf{AH}(f) \otimes \mathbf{AH}(g)} \\
\mathbf{AH}(Y_1 \times Y_2) \ar[rr]^{\nu_{Y_1 Y_2}} 
&& \mathbf{AH}(Y_1) \otimes \mathbf{AH}(Y_2)
}
$$
Moreover, any such map  $\varphi$ that makes the diagram commutative is an Adams--Hilton model for $f \times g$.
\end{enumerate}
\end{theorem}

\begin{remark}
The properties \eqref{itm:1}--\eqref{itm:4}, \eqref{itm:6}--(10) of Theorem~\ref{th_anick} are from \cite[Th.~8.1]{Anick}; the property \eqref{itm:5} follows easily from the construction of
Adams--Hilton models;
the property~(11) is  
\cite[Cor.~4.1]{AH}. 
\end{remark}

\begin{theorem}[{\cite[Th.~3.4]{AH}}]
An Adams--Hilton model for $\mathbb{C}P^2$  is given by $$\mathbf{AH}(\mathbb{C}P^2) = ( T(a_1, a_2), \partial), \quad  |a_1| =1, \quad  |a_2| = 3,  
\quad \partial a_1 = 0, \quad \partial a_2 = a_1^2 .$$
\end{theorem}

\begin{theorem} \label{theor_AH_cpn}
An Adams--Hilton model for $\mathbb{C}P^n$  is given by
\begin{gather}
\label{equat_AH_CPn}
\mathbf{AH}(\mathbb{C}P^n) = 
( T(a_1, \ldots, a_n), \partial), \quad  |a_i| =2i-1,\\
\notag
\partial a_1 = 0, 
%\quad \partial a_2 = a_1^2,
\quad
\partial a_i= a_1 \otimes a_{i-1} + a_2 \otimes a_{i-2} + \cdots +
 a_{i-1} \otimes a_1, \quad i = 2, \ldots, n.
\end{gather}
\end{theorem}

\begin{proof}
Suppose we have already constructed Adams--Hilton models $\mathbf{AH}(\mathbb{C}P^i)$ of the required form and the model
$\mathbf{AH}(\mathbb{C}P^j)$ extends the model
$\mathbf{AH}(\mathbb{C}P^i)$, for $i<j\leq n-1$.
We extend the model $\ah(\mathbb{C}P^{n-1})=(T(a_1,\ldots,a_{n-1},\partial)$ of the form \eqref{equat_AH_CPn} to some Adams--Hilton model $\ah(\mathbb{C}P^{n})=(T(a_1,\ldots,a_{n-1},a_n),\widetilde\partial)$. We have
$$ 
\widetilde\partial a_i=\partial a_i= a_1 \otimes a_{i-1} + a_2 \otimes a_{i-2} + \cdots +
 a_{i-1} \otimes a_1, \,\,i = 2, \ldots, n-1, \qquad
\widetilde{\partial} a_n = w,
$$
where $w$ is some element of degree $2n-2$. Let
$$
v:= a_1 \otimes a_{n-1} + a_2 \otimes a_{n-2} + \cdots + a_{n-1} \otimes a_1, \quad  |v| = 2n-2.
$$
We need to prove that there is an Adams--Hilton model
$\mathbf{AH}(\mathbb{C}P^n)$ with $\partial a_n = v.$
By direct computation we obtain that $\partial v = 0$ (or we can use the fact that $v$ is a cycle in the cobar construction $\cobar H_{*}(\cp)$).
Using the Hopf fibration, we get
$\Omega \mathbb{C}P^n \simeq \Omega S^{2n+1} \times S^1$, hence,
$H_{2n-2}(\Omega \mathbb{C}P^n) = 0.$ 
Therefore, $v$ is a boundary.
It follows that
$$v = k \widetilde{\partial} a_n + \partial x \quad\text{for some} \quad k \in \mathbb{Z}, \,\,
x \in T(a_1, \ldots, a_{n-1}).
$$
The differential on $T(a_1, \ldots, a_{n-1})$ increases the tensor length by one,
hence, any summand of the form  $a_i a_{n-i}$ can not appear in $\partial x$.
It follows that each summand of $v$ appears in
$k \widetilde{\partial} a_n$. 
Therefore, $ k = \pm 1.$ Then
$$
\widetilde{\partial} a_n = \pm v + \partial x, \,\,
x \in T(a_1, \ldots, a_{n-1}). 
$$
Reorienting the cell corresponding to $a_n$ if necessary, we obtain
$
\widetilde{\partial} a_n =  v + \partial x, \,\,
x \in T(a_1, \ldots, a_{n-1}). 
$
Using Proposition~\ref{prop1_AH}, we obtain that there is another model $(\ah(\C P^n),\partial)$ with 
$\partial a_n = v$ and an appropriate choice of $\theta(a_n)$, extending $(\ah(\C P^{n-1}),\partial)$.
\end{proof}

Since Adams--Hilton models respect colimits, we obtain
\begin{corollary}
An Adams--Hilton model for $\cp$  is given by
$$
  \mathbf{AH}(\cp) = ( T(a_1, a_2, \ldots), \partial), \quad |a_i| =2 i-1,
$$
$$
  \partial a_1 = 0, 
  %\quad \partial a_2 = a_1^2,
  \quad
  \partial a_i= a_1 \otimes a_{i-1} + a_2 \otimes  
  a_{i-2} + \cdots +
  a_{i-1} \otimes a_1, \quad i \geq 2,
$$
where the generator $a_i$ corresponds to the $2i$-cell in $\cp$.  
\end{corollary}

%\begin{corollary}
%The Adams--Hilton model $\mathbf{AH}(\cp)$ 
%constructed above coincides with the cobar construction
%$\cobar H_{*} (\cp)$.
%\end{corollary}

The sphere $S^n$ has the standard two-cell decomposition. The decomposition of the product $S^{n_1} \times \cdots \times S^{n_k}$ has cells of the form
$e_J = \prod\limits_{j_t \in J} e_{j_t}$, 
$J \subset \{1, \ldots, k\}$.
We denote by $b_J$ the generator corresponding to the cell~$e_J$.

\begin{theorem}[{\cite[Th.~4.3]{AH}}] \label{theor_AH_sphe}
An Adams--Hilton model 
for $X = S^{n_1} \times \cdots \times S^{n_k}$
is given~by
$$
\ah(S^{n_1} \times \cdots \times S^{n_k}) = 
(T(U), \partial), \quad
U = \langle b_J, \, J \subseteq \{1, \ldots, k\}, \,\, J \neq  \varnothing\rangle, 
$$
$$
|b_J| =  |e_J| -1 = n_{j_1} + \cdots + n_{j_s} - 1,$$
\begin{equation} \label{eq_AH_sphe1}
\partial b_J = \sum \limits_{ (I, L),  I \sqcup L = J}
(-1)^{\widetilde{\varepsilon}(I,L)} b_I b_L, 
\end{equation}
$$
\widetilde{\varepsilon}(I,L) = \sum\limits_{i \in I} {n_i} + 
\sum\limits_{i \in I, \,\,l \in L, \,\,i>l } n_i \, n_l.
$$
Moreover, this model extends the Adams--Hilton models over each
product of $\leq k - 1$ spheres.
Namely, if $K = S^{n_{j_1}} \times \cdots \times S^{n_{j_s}}$, $K \subseteq S^{n_1} \times \cdots \times S^{n_k}$, then
$$ \partial_{X} |_K = \partial_K, \,\, \theta_{X} |_K = \theta_K.$$
\end{theorem}

\begin{proposition} \label{prop_diff_AHsph}
The differential $\partial$ from Theorem~\ref{theor_AH_sphe} can be expressed in the following way:
\begin{equation} \label{eq_AH_sphe2}
\partial b_{j_1 \cdots j_s} =
\sum_{p=1}^{s-1} \sum_{\theta \in S(p, s-p)} \varepsilon(\theta)
(-1)^{|b_{j_{\theta(1)} \cdots j_{\theta(p)}}|+1 }
b_{j_{\theta(1)} \cdots j_{\theta(p)}}\otimes b_{j_{\theta(p+1)} \cdots j_{\theta(s)}} 
\end{equation} 
\begin{equation}
= \sum_{p=1}^{s-1} \sum_{\theta \in \widetilde{S}(p, s-p)} \varepsilon(\theta)
(-1)^{|b_{j_{\theta(1)} \cdots j_{\theta(p)}}|+1 }
\left[b_{j_{\theta(1)} \cdots j_{\theta(p)}}, b_{j_{\theta(p+1)} \cdots j_{\theta(s)}}\right], 
\end{equation}
where 
$1 \leq j_1 < \cdots < j_s \leq k$,
$S(p, s-p)$ denotes the set of shuffle permutations,
$\widetilde{S}(p, s-p)$ is the set of shuffle permutations such that $\theta(1)=1$, and $\varepsilon(\theta)$ is the Koszul
sign of the elements
$s b_{i_{\theta(1)}}, \ldots ,s b_{i_{\theta(s)}}.$
\end{proposition}

\begin{proof}
We prove that \eqref{eq_AH_sphe1} and \eqref{eq_AH_sphe2} are equivalent. 
Suppose $J = \{j_1, \ldots, j_s\} \subset \{ 1, \ldots, k\}, \,\,
I = \{j_{\theta(1)} \cdots j_{\theta(p)}\}, \,\,
L = \{j_{\theta(p+1)} \cdots j_{\theta(s)}\}$. It suffices to prove that the signs coincide.
First,
$$
\sum\limits_{i \in I} {n_i} = |e_I| = |b_{j_{\theta(1)} \cdots j_{\theta(p)}}|+1.
$$
Second, we have the Koszul sign $\varepsilon(\theta)$ of the elements
$s b_i, \,\,|s b_i| = n_i,  \,\,|b_i| = n_i -1,$
$$s b_{j_1} \wedge  \cdots \wedge  s b_{j_s} = \varepsilon(\theta) 
s b_{j_{\theta(1)}}  \wedge   \cdots \wedge  s b_{j_{\theta(s)}}.$$
Since $\theta$ is a shuffle permutation, it follows that
\[
\varepsilon(\theta) = \sgn \bigl({\sum\limits_{i \in I, \,\,l \in L, \,\,i>l } n_i \, n_l}\bigr).\qedhere
\]
\end{proof}

%\begin{theorem} \label{th_model_sphere}
%The Adams--Hilton model
%$\mathbf{AH}(S^{n_1} \times \cdots \times S^{n_k})$
%from Theorem~\ref{theor_AH_sphe} coincides with the cobar %construction
%$\cobar H_{*}(S^{n_1} \times \cdots \times S^{n_k})$.
%\end{theorem}

%\begin{proof}
%The corresponding tensor algebras have the same generator set, the %formulas
%for the differential are equal, compare Propositions %\ref{prop_cobar_sph} 
%and \ref{prop_diff_AHsph}.
%\end{proof}

\begin{theorem} \label{ther_AH_sph_colim}

An Adams--Hilton model  for $\sk$ is given by
$$
\mathbf{AH}(\sk) =
(T(U), \partial), \quad
U = \langle b_J, \, J \in \K, \,\, J \neq \varnothing \rangle,
\quad
|b_J| = n_{j_1} + \cdots + n_{j_s} - 1,
$$
\begin{equation} \label{eq_AH_sk}
\partial b_{j_1 \cdots j_s} =
\sum_{p=1}^{s-1} \sum_{\theta \in S(p, s-p)} \varepsilon(\theta)
(-1)^{|b_{j_{\theta(1)} \cdots j_{\theta(p)}}|+1 }
b_{j_{\theta(1)} \cdots j_{\theta(p)}}\otimes b_{j_{\theta(p+1)} \cdots j_{\theta(s)}}.
\end{equation} 
\end{theorem}

\begin{proof}
The polyhedral product $\sk$ is the colimit of products of spheres
$$
\sk = \colim_{J \in \K} (\underline{S})^{J}.
$$

Since Adams--Hilton models from Theorem~\ref{theor_AH_sphe} extend each
other over subproducts, they satisfy the coherency conditions from
Theorem~\ref{th_anick}~\eqref{itm:7}.
For example, the Adams--Hilton model 
$\mathbf{AH}(S^{n_2} \times S^{n_3})$ extends to
$\mathbf{AH}(S^{n_1} \times S^{n_2} \times S^{n_3})$ and
$\mathbf{AH}(S^{n_2} \times S^{n_3} \times S^{n_4})$. 
Therefore, the Adams--Hilton model
$\mathbf{AH}(S^{n_1} \times S^{n_2} \times S^{n_3} \cup_{S^{n_2} \times S^{n_3}} 
S^{n_2} \times S^{n_3} \times S^{n_4})$ is well defined.
The same holds for the arbitrary products of spheres. 
Thus we obtain
\[
\mathbf{AH}(\sk) = \colim_{J \in \K} \mathbf{AH}((\underline{S})^{J}).\qedhere
\]
\end{proof}

\begin{corollary} \label{cor_s2k}
An Adams--Hilton model  for $\stwok$  is given by
$$\mathbf{AH}((S^2)^{\K}) = (T(U), \partial), \quad
U = \langle b_J, \, J \in \K, \,\, J \neq \varnothing \rangle, 
$$
$$
\partial b_J = 
\sum\limits_{(I,L), I\sqcup L = J}
 b_I \otimes b_L,
 \quad 
 |b_J| = 2|J| -1.
$$
\end{corollary}

\begin{theorem} \label{th_AH_prod_cpncpk}

An Adams--Hilton model for $\mathbb{C}P^n \times \mathbb{C}P^k$ is given by
$$\mathbf{AH}(\mathbb{C}P^n \times \mathbb{C}P^k) = 
\bigl( T(\chi_\sigma\mid \sigma = 
\{\underbrace{1,\ldots,1}_{\leq n},\underbrace{2,\ldots,2}_{\leq k}\}, \,
\sigma \neq \varnothing
), 
\partial \bigr), $$
\begin{equation} \label{eq_diff_cpkcpn}
\partial \chi_\sigma =  \sum \limits_{\sigma = \tau  \sqcup \tau'}
\chi_\tau \chi_{\tau'}, \quad
|\chi_\sigma| =2 |\sigma| -1.
\end{equation}
Moreover, this model extends the models 
$\mathbf{AH}(\mathbb{C}P^s \times \mathbb{C}P^t)$,  $0 \leq s \leq n, \,
0 \leq t \leq k.$ It also extends the model
$\mathbf{AH}(S^2 \times S^2)$ from Theorem~\ref{theor_AH_sphe}.

\end{theorem}

\begin{proof}
The proof is by induction on $N = n+k$. 
The theorem holds
for $N = 2$ since
the models $\mathbf{AH}(pt \times \mathbb{C}P^2)$, 
$\mathbf{AH}(\mathbb{C}P^2 \times pt)$ and $\mathbf{AH}(S^2 \times S^2)$ 
are already constructed and satisfy the coherency conditions on
trivial intersections.

Assume the theorem holds for $N-1$. 
%Построим модель
%$\mathbf{AH}(\mathbb{C}P^n \times \mathbb{C}P^k), n+k = N.$
By induction, we have already constructed 
$\mathbf{AH}(\mathbb{C}P^{n-1} \times \mathbb{C}P^k),$ 
$\mathbf{AH}(\mathbb{C}P^n \times \mathbb{C}P^{k-1}), 
$
which extend the model
$
\mathbf{AH}(\mathbb{C}P^{n-1} \times \mathbb{C}P^{k-1})$,
the maps $\partial$ and $\theta$ agree on the intersection. 
Thus, we obtain the model on the $(2N-2)$-skeleton
$$
  sk^{2N-2} (\mathbb{C}P^n \times \mathbb{C}P^k) =
  \mathbb{C}P^{n-1} \times \mathbb{C}P^k 
  \cup_{(\mathbb{C}P^{n-1} \times \mathbb{C}P^{k-1})}
\mathbb{C}P^n \times \mathbb{C}P^{k-1},
$$
with the differential satisfying~\eqref{eq_diff_cpkcpn}.

%To obtain some model for
%$\mathbb{C}P^n \times \mathbb{C}P^k$,
%we only need to extend the maps
%$\partial$ and  $\theta$
%over
%$AH(sk^{2N-2} (\mathbb{C}P^n \times \mathbb{C}P^k) )$
%and define $\partial(\chi_L)$, 
%$\theta(\chi_{L})$, 
%$L = 
%\{\underbrace{1,\ldots,1}_{n},\underbrace{2,\ldots,2}_{k}\}$.

By Theorem~\ref{th_anick}~(10), we may choose the maps 
$\partial$ and $\theta$
on $AH(\mathbb{C}P^n \times \mathbb{C}P^k)$ 
inductively in such a way that
the ring homomorphism
$$
\nu = \nu_{\mathbb{C}P^n \mathbb{C}P^k}\colon \mathbf{AH}(\mathbb{C}P^n \times \mathbb{C}P^k)
\rightarrow
\mathbf{AH}(\mathbb{C}P^n) \otimes
\mathbf{AH}(\mathbb{C}P^k)
$$
is a chain equivalence.
We need to prove that model
$\mathbf{AH}(sk^{2N-2} (\mathbb{C}P^n \times \mathbb{C}P^k) )$ constructed above
with differential~\eqref{eq_diff_cpkcpn} satisfies
$\partial \nu = \nu \partial$.
We have
$$
\nu(\chi_{1 \cdots 1}) = \chi_{1 \cdots 1} \otimes 1, \quad
\nu(\chi_{2 \cdots 2}) = 1 \otimes \chi_{2 \cdots 2} , \quad
\nu(\chi_{1 \cdots 1 2 \cdots 2}) =0.
$$
Hence,
\[
  \partial \nu (\chi_{1\cdots 1}) = \partial (\chi_{1\cdots 1} \otimes 1) = 
\partial (\chi_{1\cdots 1}) \otimes 1 = \nu \partial  (\chi_{1\cdots 1})
\]
and similarly for $\chi_{2\cdots 2}$. If $\sigma$ contains both $1$ and $2$, then  
$$
\nu \partial (\chi_\sigma) = \nu (\sum \chi_\tau \chi_{\tau'}) =
\nu (\chi_{1 \cdots 1} \chi_{2 \cdots 2} +
\chi_{2 \cdots 2} \chi_{1 \cdots 1}) = 0 = \partial \nu (\chi_\sigma).
$$
By Theorem~\ref{th_anick}~(10), there is an Adams--Hilton model
$(AH(\mathbb{C}P^n \times \mathbb{C}P^k),\partial)$ extending the model $\mathbf{AH}(sk^{2N-2} (\mathbb{C}P^n \times \mathbb{C}P^k) )$
and satisfying $\nu \partial = \partial \nu$. 

Consider the generator $\chi_{\widetilde{\sigma}}$,
${\widetilde{\sigma}} = 
\{\underbrace{1,\ldots,1}_{n},\underbrace{2,\ldots,2}_{k}\}$
corresponding to the maximal cell of $\mathbb{C}P^n \times \mathbb{C}P^k$. To complete the proof, we need to modify the differential $\partial$ on $AH(\mathbb{C}P^n \times \mathbb{C}P^k)$ by changing it on $\chi_{\widetilde{\sigma}}$ only, so that the new differential $\widetilde\partial$ satisfies
$$
\widetilde\partial(\chi_{\widetilde{\sigma}})= \sum\limits_{\tau, \, \tau', \, \tau\sqcup \tau' = {\widetilde{\sigma}}}\chi_\tau \chi_{\tau'}.
$$
Denote the right hand side of the identity above by $x$. Then
$$
\nu x = \nu (\sum 
\chi_\tau \chi_{\tau'}) = \nu (\chi_{1 \cdots 1} \chi_{2 \cdots 2} +
\chi_{2 \cdots 2} \chi_{1 \cdots 1}) = 0.
$$
Also,
$\partial x = 0$ by direct computation. Since $\nu$ is a chain equivalence, $x$ is a boundary. 
Then 
$$
x = k \partial(\chi_{\widetilde{\sigma}}) + \partial y,
\quad y \in 
\mathbf{AH}(sk^{2N-2} (\mathbb{C}P^n \times \mathbb{C}P^k) ).
$$
The differential $\partial$ in
$\mathbf{AH}(sk^{2N-2} (\mathbb{C}P^n \times \mathbb{C}P^k) )$ 
preserves multisets and increases the tensor length by one. 
%Therefore, the summand $\chi_\tau \chi_{\tau'}, \tau\sqcup \tau' = {\widetilde{\sigma}}$, we have in the expression of the element $x$ can not appear in $\partial y$. 
%It follows that each summand of this form can only appear in $\widetilde{\partial}(\chi_{\widetilde{\sigma}})$, 
%thus $k = \pm 1$. 
%By changing the orientation if necessary, we have
%$k = 1.$
As in the proof of Theorem~\ref{theor_AH_cpn}, we have $k = 1$.
Now Proposition~\ref{prop1_AH} implies that we may change $\partial$ to $\widetilde\partial$ satisfying $\widetilde\partial \chi_{\widetilde{\sigma}} = x$.% is a valid choice of $\partial \chi_{\widetilde{\sigma}}$.
\end{proof}

\begin{theorem}

An Adams--Hilton model for
$\mathbb{C}P^{k_1} \times \cdots \times \mathbb{C}P^{k_n}$ is given by
$$\mathbf{AH}(\mathbb{C}P^{k_1} \times \cdots \times \mathbb{C}P^{k_n}) = 
\bigl( T(\chi_\sigma\mid 
\sigma = 
\{\underbrace{1,\ldots,1}_{\leq k_1},\underbrace{2,\ldots,2}_{\leq k_2},
  \ldots,\underbrace{n,\ldots,n}_{\leq k_n}\}, \,
\sigma \neq \varnothing
), 
\partial \bigr), $$
\begin{equation} \label{eq_diff_cpkcpncps}
\partial \chi_\sigma =  \sum \limits_{\sigma = \tau  \sqcup \tau'}
\chi_\tau \chi_{\tau'}, \quad
|\chi_\sigma| =2 |\sigma| -1
.
\end{equation}
Moreover, this model extends models 
$\mathbf{AH}(\mathbb{C}P^{s_1} \times \cdots \times \mathbb{C}P^{s_n}),$
$0 \leq s_i \leq n_i$.
It also extends the model
$\mathbf{AH}(S^2 \times \cdots \times S^2)$ from Theorem~\ref{theor_AH_sphe}.

\end{theorem}

\begin{proof}
Analogously, the proof is by induction on $N = k_1 + \cdots + k_n$. 
The case $N = 2$ is similar to the one from the previous theorem.

Assume the theorem holds for $N-1$.
%$(\mathbb{C}P^{k_1} \times \cdots \times \mathbb{C}P^{k_n})$
%for all  $(k_1, \ldots, k_n)$ sati $k_1 +N-1$. 
By induction, we have already constructed models
$$
\mathbf{AH}(\mathbb{C}P^{k_1 -1} \times 
\mathbb{C}P^{k_2} \times 
\cdots \times \mathbb{C}P^{k_n}),
 \ldots, 
\mathbf{AH}(\mathbb{C}P^{k_1} \times 
\mathbb{C}P^{k_2} \times 
\cdots \times \mathbb{C}P^{k_n -1}),
$$
the maps $\partial$ and $\theta$ agree on intersections. 

Thus we obtain the model on $(2N-2)$-skeleton
$$
sk^{2N-2} (\mathbb{C}P^{k_1} \times \cdots \times \mathbb{C}P^{k_n}) =
\mathbb{C}P^{k_1-1} \times \cdots \times \mathbb{C}P^{k_n}
\cup
\cdots \cup
\mathbb{C}P^{k_1 } \times \cdots \times \mathbb{C}P^{k_n -1},
$$
the formula \eqref{eq_diff_cpkcpncps} holds.

By Theorem~\ref{th_anick}~(10), we may choose the maps 
$\partial$ and $\theta$
on $AH(\mathbb{C}P^{k_1} \times \cdots \times \mathbb{C}P^{k_n})$ 
inductively in such a way that
the ring homomorphism
$$
\nu \colon \mathbf{AH}(\mathbb{C}P^{k_1} \times \cdots \times \mathbb{C}P^{k_n})
\rightarrow
\mathbf{AH}(\mathbb{C}P^{k_1} \times \cdots \times \mathbb{C}P^{k_{n-1}}) \otimes
\mathbf{AH}(\mathbb{C}P^{k_n})
$$
is a dga-morphism, and hence a chain equivalence.
We prove that the constructed model
$\mathbf{AH}(sk^{2N-2} (\mathbb{C}P^{k_1} \times \cdots \times \mathbb{C}P^{k_n}) )$ satisfies
$\partial \nu = \nu \partial$.
We have
\begin{multline*}
\nu(\chi_{\sigma}) = \chi_{\sigma} \otimes 1, \,
 n \notin \sigma;
\\
\nu(\chi_{\sigma}) = 1 \otimes \chi_{\sigma} , 
\, \sigma = \{n, \ldots, n\};
\\
\nu(\chi_{\sigma}) =0, \,\text{otherwise}.
\end{multline*}
Suppose $n \notin \sigma$ , then
$$\partial \nu (\chi_\sigma) = \partial (\chi_\sigma \otimes 1) = 
\partial (\chi_\sigma) \otimes 1 = \nu \partial  (\chi_\sigma).$$
Suppose $n \in \sigma$, $i \in \sigma$, $i = 1, \ldots, n-1$, then  
$
\sigma = \alpha \sqcup \beta, \,\,
n \notin \alpha, \,\,
\beta = \{n, \ldots, n\},
$
$$
\nu \partial (\chi_\sigma) = \nu (\sum \chi_\tau \chi_{\tau'}) = \text{// the rest summands vanish //}
$$
$$
=
\nu (\chi_{\alpha} \chi_{\beta} +
\chi_{\beta} \chi_{\alpha}) = 0 = \partial \nu (\chi_\sigma).
$$
By Theorem~\ref{th_anick}~(10), there is an Adams--Hilton model
$\mathbb{C}P^{k_1} \times \cdots \times \mathbb{C}P^{k_n}$ 
that extends the model $\mathbf{AH}(sk^{2N-2} (\mathbb{C}P^{k_1} \times \cdots \times \mathbb{C}P^{k_n}) )$,
for which the equality $\nu \partial = \partial \nu$ holds. 

Consider the generator $\chi_{\widetilde{\sigma}}$,
${\widetilde{\sigma}} = 
\{\underbrace{1,\ldots,1}_{k_1},\underbrace{2,\ldots,2}_{k_2},
  \ldots,\underbrace{n,\ldots,n}_{k_n}\}$.
Denote by $\widetilde{\partial}(\chi_{\widetilde{\sigma}})$
the differential in the extended model  
$\mathbf{AH}(\mathbb{C}P^{k_1} \times \cdots \times \mathbb{C}P^{k_n})$.

Denote by $x$ the element
$$
x:= \sum\limits_{\tau, \, \tau', \, \tau\sqcup \tau' = {\widetilde{\sigma}}}
\chi_\tau \chi_{\tau'}.
$$
Then
$$
\nu x = \nu (\sum 
\chi_\tau \chi_{\tau'}) = \nu (\chi_{\alpha} \chi_{\beta} +
\chi_{\beta} \chi_{\alpha}) = 0,
$$
where
$
\sigma = \alpha \sqcup \beta, \,\,
n \notin \alpha, \,\,
\beta = \{n, \ldots, n\}
$.
By direct computation we obtain
$\partial x = 0$. The map $\nu$ is a chain equivalence,
hence an element $x$ is a boundary. 
Then 
$$
x = k \widetilde{\partial}(\chi_{\widetilde{\sigma}}) + \partial y,
\quad y \in 
\mathbf{AH}(sk^{2N-2} (\mathbb{C}P^{k_1} \times \cdots \times \mathbb{C}P^{k_n} ).
$$
To conclude the proof it remains to use the argument from the previous theorem. 
\end{proof}

\begin{corollary} \label{cor_AH_cpinf}

An Adams--Hilton model 
for the product of $n$ copies of $\cp$
is given~by
$$\mathbf{AH}(\cp \times \cdots \times \cp) = 
\bigl( T(\chi_\sigma\mid 
\sigma = 
\{1,\ldots,1,2,\ldots,2,
  \ldots,n,\ldots,n\}, \,
\sigma \neq \varnothing
), 
\partial \bigr), $$
$$
\partial \chi_\sigma =  \sum \limits_{\sigma = \tau  \sqcup \tau'}
\chi_\tau \chi_{\tau'}, 
\quad
|\chi_\sigma| =2 |\sigma| -1
.
$$
Moreover,
this model extends the model
$\mathbf{AH}(S^2 \times \cdots \times S^2)$ from Theorem~\ref{theor_AH_sphe}.

\end{corollary} 

\begin{theorem} \label{th_AH_cpk}

An Adams--Hilton model for $\cpk$ 
is given by
$$\mathbf{AH}(\cpk) = 
\bigl( T(\chi_\sigma\mid 
I_{\sigma} \in \K, \,\sigma \neq \varnothing
), 
\partial \bigr), $$
\begin{equation} 
\partial \chi_\sigma =  \sum \limits_{\sigma = \tau  \sqcup \tau'}
\chi_\tau \chi_{\tau'}, 
\quad
|\chi_\sigma| =2 |\sigma| -1
,
\end{equation}
where $I_{\sigma}$ denotes the support of a multiset $\sigma$.
Moreover,
this model extends the  model
$\mathbf{AH}(\stwok)$ from Corollary~\ref{cor_s2k}.
\end{theorem} 

\begin{proof}
The polyhedral product $\cpk$ is the colimit of products 
$$
\cpk = \colim_{J \in \K} (\cp)^{J}.
$$

The Adams--Hilton models from Corollary~\ref{cor_AH_cpinf} extend each
other over subproducts and therefore satisfy the coherency conditions, see
Theorem~\ref{th_anick}~\eqref{itm:7}.
Thus, we obtain
\[
\mathbf{AH}(\cpk) = \colim_{J \in \K} \mathbf{AH}((\cp)^{J}).
\qedhere
\]
\end{proof}

\begin{theorem}  \label{th_AH_cobar_coincide}
Let $\K$ be an arbitrary simplicial complex.
Consider Adams--Hilton models $\ah(\sk)$ and $\ah(\cpk)$
from Theorem~\ref{ther_AH_sph_colim} and 
Theorem~\ref{th_AH_cpk}. Then 
$$
  \ah(\sk)  \cong \cobar H_{*}(\sk),\qquad
  \ah(\cpk)  \cong \cobar H_{*}(\cpk)
$$
in the category DGA.
\end{theorem}

%\begin{theorem} \label{th_sk_coincide}
%The Adams--Hilton model
%$\mathbf{AH}(\sk)$
%from Theorem~\ref{ther_AH_sph_colim} coincides with the cobar %construction
%$\cobar H_{*}(\sk)$.
%\end{theorem}
%\begin{theorem} \label{th_cpk_coincide}
%The Adams--Hilton model
%$\mathbf{AH}(\cpk)$
%from Theorem~\ref{th_AH_cpk} coincides with the cobar construction
%$\cobar H_{*}(\cpk)$.
%\end{theorem}

\begin{proof}
For the case $\sk$, this follows by comparing Theorem~\ref{ther_AH_sph_colim} with Proposition~\ref{prop_cobar_sph}.
For the case $\cpk$, 
compare Theorem~\ref{th_AH_cpk} and formulae~\eqref{cobar_cpk},~\eqref{diff_cpk}.
\end{proof}

\begin{remark}
Any Adams--Hilton model has a structure of a "Hopf algebra up to homotopy".
For details, see \cite{Anick}.
\end{remark}

\begin{remark}
In general, an Adams--Hilton model $\ah(X) $ does not coincide with the cobar construction of the homology coalgebra. For example, the model for Moore space from {\cite[Cor.~2.3]{AH} } is isomorphic to $\cobar (C,\partial)$ with nontrivial differential $\partial$. 
\end{remark}

%\textcolor{red}{
%\begin{remark}
%Note that $\ah(X)$ does not coincide with any cobar construction of %coalgebra generally.
%\end{remark}
%}

\section{Adams--Hilton models for maps and chains in cobar construction corresponding to Whitehead products}

Here we describe the chain in the cobar construction
$\cobar H_{*}(\cpk)$ corresponding to the Hurewicz image  of the higher Whitehead product $[[\mu_1, \mu_2, \mu_3], \mu_4, \mu_5]$. We consider the Adams--Hilton models 
for
$\sk$, $\stwok$, $\cpk$ from 
Theorem~\ref{ther_AH_sph_colim}, 
Corollary~\ref{cor_s2k},
Theorem~\ref{th_AH_cpk}, which coincide with 
the corresponding cobar constructions. The reason we use Adams--Hilton models is that they behave nicely with respect to CW-maps.

%Here we consider the Adams--Hilton models for
%$\sk$, $\stwok$, $\cpk$ from 
%Theorem~\ref{ther_AH_sph_colim}, 
%Corollary~\ref{cor_s2k},
%Theorem~\ref{th_AH_cpk}. These models coincide with the corresponding cobar constructions.

\begin{proposition} \label{prop_incl_sk_cpk}
Let $\K$ be an arbitrary simplicial complex.
Then the inclusion $\ah(\stwok) \hookrightarrow \ah(\cpk)$ is an Adams--Hilton model 
$\ah(i)$
for the inclusion $\stwok \xhookrightarrow{i} \cpk$.
\end{proposition}

\begin{proof}
By Theorem~\ref{th_AH_cpk}, the model $\ah(\cpk)$ extends the model 
$\ah(\stwok)$. Now the result follows from Theorem~\ref{th_anick}~\eqref{itm:5}.
\end{proof}

\begin{corollary}
Let $\K$ be an arbitrary simplicial complex.
Then the inclusion $\stwok \hookrightarrow \cpk$
and the inclusion of dgas
$$\cobar(H_{*}(\stwok)) \hookrightarrow \cobar(H_{*}(\cpk))$$
induce the same map 
$$
H_{*}(\Omega \stwok) \rightarrow H_{*}(\Omega \cpk )
$$
in homology groups.
\end{corollary}

\begin{proposition} \label{prop_AHmap_attaching}
Let
$$\omega \colon  S^{N-1} \rightarrow (\underline{S})^{\partial \Delta (1, \ldots, k)} \cong T(S^{n_1}, \ldots, S^{n_k}), \,\,N = n_1 + \cdots + n_k,$$
be the attaching map of the top cell
in $S^{n_1}\times \cdots \times S^{n_k}$.
Suppose
$\ah( (\underline{S})^{\partial \Delta })$ is the Adams--Hilton model
from Theorem~\ref{ther_AH_sph_colim}.
Then there is an Adams--Hilton model $\ah(\omega)$ given~by
$$\ah(\omega) \colon  \ah(S^{N-1}) \cong (T(a), 0) \longrightarrow
\ah( T(S^{n_1}, \ldots, S^{n_k})),$$
$$
\ah( T(S^{n_1}, \ldots, S^{n_k}))
\cong
(T(b_J, J \in \partial \Delta, J \neq \varnothing), \partial),
$$
$$
a \mapsto
\sum_{p=1}^{k-1} \sum_{\theta \in S(p, k-p)} \varepsilon(\theta)
(-1)^{|b_{{\theta(1)} \cdots {\theta(p)}}|+1 }
b_{{\theta(1)} \cdots {\theta(p)}}\otimes b_{{\theta(p+1)} 
\cdots {\theta(k)}}. 
$$
\end{proposition}

\begin{remark}
The element $a$ maps exactly to the element $\partial b_{1 \cdots k}$, where
$b_{1 \cdots k}$ corresponds to the top cell in the product of spheres.
\end{remark}

\begin{proof}
Suppose $X$ is a CW complex, for which the $(n+1)$-skeleton $X^{n+1}$ is obtained from
$X^n$ by attaching  a single $(n+1)$-dimensional cell $e$, and let $v$ be the corresponding generator in $\ah(X^{n+1})$.
Let $f \colon  S^n \rightarrow X^n$ be the attaching map of~$e$. An Adams--Hilton model for $f$ makes the following diagram commutative up to chain homotopy:
$$
\xymatrix{
\mathbf{AH}(S^n) \ar[r] \ar[d]^{\mathbf{AH}(f)} & CU_{*}(\Omega S^n) \ar[d]^{CU_{*}(\Omega f)} \\
\mathbf{AH}(X^n) \ar[r]^{\theta_{X^n}} & CU_{*}(\Omega X^n)
}
$$
Since $\mathbf{AH}(S^n) \cong (T(a), 0)$, it is enough to define $\ah(f)$ on $a$. Suppose $\ah(f)(a) = z$, $[z]\in H_*(\Omega X^n)$. From the diagram above we get
$$
(\theta_{X^n})_{*} [z] = H(\Omega f) [a].
$$
From the definition of the differential $\partial_{X^{n+1}}$ (see formula~\eqref{eq_def_AH_dif}) we obtain
$\partial_{X^{n+1}}(v) = z+\partial_{X^n} b$ with some $b\in\ah(X^n)$. By Proposition~\ref{prop2_AH_f}, we may change $\ah(f)$ by adding a boundary and obtain $\ah(f)(a)=\partial_{X^{n+1}}(v)$. 

Applying the construction above
to the map
$\omega \colon  S^{N-1} \rightarrow (\underline{S})^{\partial \Delta (1, \ldots, k)} \cong T(S^{n_1}, \ldots, S^{n_k})$, we conclude that
we can take $\ah(\omega)(a) = \partial b_{1 \cdots k}$, where $b_{1 \cdots k}$ is the generator corresponding to the top cell of $S^{n_1} \times \cdots \times S^{n_k}$.
\end{proof}

\begin{theorem} \label{th_prod_sph_AH}
Consider the product map
$$
f \colon  S^5 \times S^2 \rightarrow T(S^2, S^2, S^2) \times S^2 \cong (S^2)^{\partial \Delta(1,2,3) * \{4\}}
$$
of the attaching map $S^5 \rightarrow T(S^2, S^2, S^2)$ and the identity map $S^2 \rightarrow S^2$.
Then there is an Adams--Hilton model $\ah(f)$ given by
$$
\ah(S^5 \times S^2) \cong T(b_1, b_2, b_{12}), \quad |b_1|= 4,\quad
|b_2| = 1, \quad |b_{12}| = 6,
$$
$$
\partial b_i = 0, \quad  \partial b_{12} = - [b_1, b_2];
$$
$$
\ah(T(S^2, S^2, S^2) \times S^2) \cong T(c_J, \,J \in \partial \Delta (1,2,3) * \{4\}, \,J \neq \varnothing), 
$$
$$
\partial c_J = 
\sum\limits_{(I,L), I\sqcup L = J}
 c_I \otimes c_L;
$$
$$ 
b_1 \xmapsto{\ah(f)} ``\partial c_{123}\mbox{''} = 
[c_1, c_{23}] + [c_{12}, c_3] + [c_{13}, c_2], 
\quad 
b_2 \xmapsto{\ah(f)} c_4,
$$
$$
b_{12} 
\xmapsto{\ah(f)}
 [ c_{124}, c_3] +[ c_{134}, c_2] 
+[ c_1, c_{234}] + [ c_{12}, c_{34}] 
+ [ c_{13}, c_{24}] + [ c_{14}, c_{23}].$$ 
\end{theorem}

\begin{proof}
Adams--Hilton models for $S^5 \times S^2$ and 
$T(S^2, S^2, S^2) \times S^2$ are described in Theorem~\ref{ther_AH_sph_colim}
and Corollary~\ref{cor_s2k}.
It is easy to see that the maps of algebras defined in 
Theorem~\ref{th_anick}~(10)
$$
\ah(S^5 \times S^2) \xrightarrow{\nu_{S^5 S^2}} \ah(S^5) \otimes \ah(S^2),
$$
$$
\ah(T(S^2, S^2, S^2) \times S^2) \xrightarrow{\nu_{T S^2}} 
\ah(T(S^2, S^2, S^2)) \otimes \ah( S^2)
$$
are dga-morphisms, i.e. commute with the differentials.
The model for the attaching map
$\omega \colon  S^5 \rightarrow T(S^2, S^2, S^2)$ is given by Proposition~\ref{prop_AHmap_attaching}:
$$
\ah(S^5) \cong (T(b_1), 0), \quad |b_1| = 4,
$$
$$
\ah(T(S^2, S^2, S^2)) \cong T(c_1, c_2, c_3, c_{12}, c_{13}, c_{23}),
\quad  |c_i| = 1, \quad |c_{ij}| = 3,
$$
$$
\partial c_i = 0, \quad \partial c_{ij} = [c_i, c_j] = c_i \otimes c_j + c_j \otimes c_i,
$$
$$
b_1 \xmapsto{\ah(\omega)}
``\partial c_{123}\mbox{''} = 
[c_1, c_{23}] + [c_{12}, c_3] + [c_{13}, c_2].
$$
The identity map $S^2 \rightarrow S^2$ is modeled by the identity 
map of the Adams--Hilton models:
$$
\ah(S^2) \cong (T(b_2), 0) \rightarrow \ah(S^2) \cong (T(c_4), 0),\quad  b_2 \mapsto c_4.
$$
By Theorem~\ref{th_anick}~(11), a map 
$$
\varphi\colon  \ah(S^5 \times S^2) \rightarrow \ah(T(S^2, S^2, S^2) \times S^2)
$$
is an Adams--Hilton model for the product of maps
if it makes the following diagram commutative
$$
\xymatrix{
\ah(S^5 \times S^2) \ar[rr]^{\nu_{S^5 S^2}} 
\ar[d]^{\varphi} & & \ah(T(S^2, S^2, S^2)) \otimes \ah(S^2)
\ar[d]^{\mathbf{AH}(\omega) \otimes 1_{\mathbf{AH(S^2)}}} \\
\ah(T(S^2, S^2, S^2) \times S^2) \ar[rr]^{\nu_{T S^2}} 
& &\mathbf{AH}(T(S^2, S^2, S^2)) \otimes \mathbf{AH}(S^2)
}
$$
Since $\ah(S^5 \times S^2)$ extends the model $\ah(S^5 \vee S^2)$ and
$\ah(S^5 \times S^2) \cong (T(b_1, b_2, b_{12}), \partial)$ we only need to define $\varphi$ on $b_{12}$. We put
$$
b_{12} 
\xmapsto{\varphi}
 [ c_{124}, c_3] +[ c_{134}, c_2] 
+[ c_1, c_{234}] + [ c_{12}, c_{34}] 
+ [ c_{13}, c_{24}] + [ c_{14}, c_{23}]$$ 
and check that $\varphi$ commutes with the differential.
We have
$$
b_{1} \xmapsto{\varphi} ``\partial c_{123}\mbox{''},\quad  
b_2\xmapsto{\varphi} c_4, \quad 
\partial b_{12} = - [b_1, b_2] \xmapsto{\varphi} 
-[``\partial c_{123}\mbox{''}, c_4].
$$
Using the inclusion of models $\ah(T(S^2, S^2, S^2) \times S^2)
\hookrightarrow \ah((S^2)^{\Delta(1,2,3,4)})$ we get
$$
\partial c_{1234} = [c_{123}, c_4] +
[ c_{124}, c_3] +[ c_{134}, c_2] 
+[ c_1, c_{234}] + [ c_{12}, c_{34}] 
+ [ c_{13}, c_{24}] + [ c_{14}, c_{23}].
$$
By applying $\partial$ once again we obtain
$$
\partial[c_{123}, c_4] = [\partial c_{123}, c_4] = 
- \partial ([ c_{124}, c_3] +[ c_{134}, c_2] 
+[ c_1, c_{234}] + [ c_{12}, c_{34}] 
+ [ c_{13}, c_{24}] + [ c_{14}, c_{23}]),
$$
where both sides are now defined in $\ah(T(S^2, S^2, S^2) \times S^2)$.
Note that
\begin{multline} \label{eq_c1234}
b_{12} \xmapsto{\varphi}
``\partial c_{1234} - [c_{123}, c_4]\mbox{''} 
\\
=[ c_{124}, c_3] +[ c_{134}, c_2] 
+[ c_1, c_{234}] + [ c_{12}, c_{34}] 
+ [ c_{13}, c_{24}] + [ c_{14}, c_{23}].
\end{multline}

Here we put quotation marks around $\partial c_{1234} - [c_{123}, c_4]$  because the elements $c_{1234}$ and $[c_{123}, c_4]$ are defined in $\ah((S^2)^{\Delta(1,2,3,4)})$ only, while $\partial c_{1234} - [c_{123}, c_4]$ is in $\ah(T(S^2, S^2, S^2) \times S^2)$.

It remains to check that $\varphi$ makes the diagram above commutative.
We have $\nu_{S^5 S^2} (b_{12})=0$ by the definition of $\nu_{S^5 S^2}$. Since
$\nu_{T S^2}$ is a ring homomorphism, it commutes with the Lie bracket, and we obtain
$\nu_{T S^2}\varphi(b_{12}) = 0$.
\end{proof}

\begin{theorem} \label{th_map_colim_AH}
Let $\K = \partial \Delta (\partial \Delta (1,2,3), 4, 5)$.
Consider the map
$$
g \colon  T(S^5, S^2, S^2) \rightarrow \stwok,
$$
induced by the maps $S^5 \xrightarrow{\omega} T(S^2, S^2, S^2)$, \,$S^2 \xrightarrow{1_{S^2}} S^2$.

Then there is an Adams--Hilton model $\ah(g)$ given by
$$
\ah(T(S^5, S^2, S^2)) \cong T(b_1, b_2, b_3, b_{12}, b_{13}, b_{23}), 
$$
$$
|b_1|= 4,\quad 
|b_2|= |b_3| = 1, \quad |b_{12}| = |b_{13}| = 6,\quad  |b_{23}| =3,
$$
$$
\partial b_i = 0, \quad  \partial b_{12} = - [b_1, b_2],
\quad  \partial b_{13} = -[b_1, b_3], \quad  \partial b_{23} = [b_2, b_3]; 
$$
$$
\ah(\stwok) \cong T(c_J, \,J \in  \K, \,J \neq \varnothing), 
$$
$$
\partial c_J = 
\sum\limits_{(I,L), I\sqcup L = J}
 c_I \otimes c_L;
$$
\begin{equation} \label{eq_c123}
b_1 \xmapsto{\ah(g)} ``\partial c_{123}\mbox{''} = 
[c_1, c_{23}] + [c_{12}, c_3] + [c_{13}, c_2], 
\quad 
b_2 \xmapsto{\ah(g)} c_4,
\quad 
b_3 \xmapsto{\ah(g)} c_5,
\end{equation}
$$
b_{12} 
\xmapsto{\ah(g)}
 [ c_{124}, c_3] +[ c_{134}, c_2] 
+[ c_1, c_{234}] + [ c_{12}, c_{34}] 
+ [ c_{13}, c_{24}] + [ c_{14}, c_{23}],
$$ 
$$
b_{13} 
\xmapsto{\ah(g)}
 [ c_{125}, c_3] +[ c_{135}, c_2] 
+[ c_1, c_{235}] + [ c_{12}, c_{35}] 
+ [ c_{13}, c_{25}] + [ c_{15}, c_{23}],
$$ 
$$
b_{23} 
\xmapsto{\ah(g)} c_{45}
.
$$ 
\end{theorem}

\begin{proof}
The result follows Theorem~\ref{th_anick}~\eqref{itm:8}, because the map $g$ is the colimit of two maps $f$ from Theorem~\ref{th_prod_sph_AH} and the identity map.
\end{proof}

\begin{theorem} \label{th_wh}
Let $\K = \partial \Delta (\partial \Delta (1,2,3), 4,5)$.
Let $[[\mu_1, \mu_2, \mu_3], \mu_4, \mu_5] \in \pi_7(\Omega \cpk)$
be the canonical iterated Whitehead product in $\cpk$, given by the composite
$$
S^8 \xrightarrow{\,\,\, \omega \,\,\,} T(S^5, S^2, S^2) 
\xrightarrow{\,\,\, g \,\,\,} (S^2)^{\K}
\xrightarrow{\,\,\, i \,\,\,} \cpk.
$$
Then Adams--Hilton models can be constructed explicitly for each map in the sequence:
$$
\mathbf{AH}(S^8) \longrightarrow \mathbf{AH}(T(S^5, S^2, S^2)) \longrightarrow 
\mathbf{AH}((S^2)^{\K}) \longrightarrow \mathbf{AH}(\cpk).
$$
Furthermore, the Hurewicz image of the higher Whitehead product
$[[\mu_1, \mu_2, \mu_3], \mu_4, \mu_5]$ is represented by the following cycle in the cobar construction 
$\cobar H_{*}(\cpk) \cong \mathbf{AH}(\cpk)$:
$$
- \partial \bigl( 
[\chi_{123}, \chi_{45}] + [\chi_{1234}, \chi_5] + [\chi_{1235}, \chi_4] 
\bigr) .
$$
\end{theorem}

\begin{proof}
Adams--Hilton models for $S^8$, $T(S^5, S^2, S^2)$, $\stwok$ and $\cpk$ are
constructed in Theorem~\ref{ther_AH_sph_colim}, Corollary~\ref{cor_s2k} and
Theorem~\ref{th_AH_cpk}. Adams--Hilton models for the maps $\omega$, $g$ and $i$
are constructed in Proposition~\ref{prop_AHmap_attaching}, Theorem~\ref{th_map_colim_AH} and Proposition~\ref{prop_incl_sk_cpk}.
By Theorem~\ref{th_anick}~\eqref{itm:3} we
may take $\ah(i \circ g \circ \omega)$ as an
Adams--Hilton model for the iterated higher Whitehead product 
$i \circ g \circ \omega = [[\mu_1, \mu_2, \mu_3], \mu_4, \mu_5]$.

It remains to calculate the image of $\ah(i \circ g \circ \omega)$ on the generator of $\ah(S^8)$. We have
$$
\ah(S^8) \cong (T(a_1), 0), \quad a_1 \xmapsto{\ah({\omega})}
``\partial b_{123}\mbox{''} = -[b_1, b_{23}]-[b_{12}, b_3]-[b_{13}, b_2].
$$
We use formulae \eqref{eq_c1234}, \eqref{eq_c123}, the inclusions of models 
$\ah(\stwok)
\hookrightarrow \ah((S^2)^{\Delta^4})$ and
$\ah(\cpk)
\hookrightarrow \ah((\cp)^{\Delta^4})$ to calculate
$$
``\partial b_{123}\mbox{''}
\xmapsto{\ah(g)}
- [``\partial c_{123} \mbox{''}, c_{45}] 
- [``\partial c_{1234} \mbox{''}, c_5] 
+[[``c_{123} \mbox{''}, c_4], c_5] 
- [``\partial c_{1235} \mbox{''}, c_4] 
+ [[``c_{123} \mbox{''}, c_5], c_4]
$$
$$
=
- [``\partial c_{123} \mbox{''}, c_{45}] 
- [``\partial c_{1234} \mbox{''}, c_5] 
-[[c_4, c_5], ``c_{123} \mbox{''}] 
- [``\partial c_{1235} \mbox{''}, c_4] 
$$
$$
=
- \partial \left([`` c_{123} \mbox{''}, c_{45}] 
+ [`` c_{1234} \mbox{''}, c_5] 
+ [``c_{1235} \mbox{''}, c_4] \right)
$$
$$
\xmapsto{\ah(i)}
- \partial \left([`` \chi_{123} \mbox{''}, \chi_{45}] 
+ [`` \chi_{1234} \mbox{''}, \chi_5] 
+ [``\chi_{1235} \mbox{''}, \chi_4] \right).
$$
The elements $\chi_{123}$, $\chi_{1234}$ and $\chi_{1235}$ are not defined in $\cobar H_{*}(\cpk)$, but the boundary of the element above belongs there.
\end{proof}


\begin{thebibliography}{99}



%+
\bibitem{Ab}
S. Abramyan.
\emph{Iterated higher Whitehead products in topology of moment-angle complexes}.
Sibirsk. Mat. Zh.~\textbf{60} (2019), no. 2, 243–256 (Russian); Siberian Math. J.~\textbf{60} (2019), no. 2, 185–196 (English translation).

%+
\bibitem{AbPa}
S. Abramyan, T. Panov.
\emph{Higher Whitehead Products in Moment–Angle Complexes and Substitution of Simplicial Complexes}.
Tr. Mat. Inst. Steklova, ~\textbf{305} (2019),  7–28. 


%+
\bibitem{Adams}
J. F. Adams.
\emph{On the cobar construction}. Proc. Nat. Acad. Sci. U.S.A.~\textbf{42} (1956), no. 7, 409--412.


%+
\bibitem{AH}
J. F. Adams, P. J. Hilton.
\emph{On the chain algebra of a loop space}.
Comment. Math. Helv.~\textbf{30} (1956), 305-330.

%+
\bibitem{arc}
P. Andrews, M. Arkowitz.
\emph{Sullivan’s minimal models and higher order Whitehead products}. Can. J.Math.~\textbf{30} (1978), no. 5, 961--982.

%+
\bibitem{Anick}
D. Anick.
\emph{Hopf algebras up to homotopy}.
Journal of the American Mathematical Society.~\textbf{2} (1989), 417--453.

%+
\bibitem{BBCG}
A. Bahri, M. Bendersky, F.R. Cohen, S. Gitler.
\emph{The polyhedral product functor: a method of computation for moment-angle complexes, arrangements and related spaces}. Adv. Math.~\textbf{225} (2010), no. 3, 1634--1668.

%+
\bibitem{BBCG2}
A. Bahri, M. Bendersky, F.R. Cohen, S. Gitler.
\emph{On problems
concerning moment-angle complexes, and polyhedral products}. 
Tr. Mosk. Mat. Obs.~\textbf{74} (2013), no. 2, 247–-264 (Russian);
Trans. Moscow Math. Soc.,~\textbf{74} (2013), 203–216 (English translation).


%+
\bibitem{tortop}
V. Buchstaber, T. Panov. \emph{Toric topology}. Math.
Surv. and Monogr.,~\textbf{204}. Amer. Math. Soc., Providence, RI, 2015. 

%??
%\bibitem{belchi}
%F. Belchi, U. Buijs, J. M. Moreno-Fernandez, A. Murillo.
%\emph{Higher order Whitehead products and L-infinity
%structures on the homology of a DGL.} 
%Linear Algebra and Its Applications~520 (), 16--31.

%+
\bibitem{GIPS}
J. Grbi\'c, M. Ilyasova, T. Panov, G. Simmons.
\emph{One-relator groups and algebras related to polyhedral products.} Proc. Roy. Soc. Edinburgh Sect. A, to appear; DOI:10.1017/prm.2020.101; arXiv:2002.11476.


%+
\bibitem{GPTW}
J. Grbi\'c, T. Panov, S. Theriault, J. Wu. 
\emph{The homotopy types of moment-angle complexes for
flag complexes}. Trans. Amer. Math. Soc.~\textbf{368} (2016), no. 9, 6663-–6682.

%+
\bibitem{GT}
J. Grbi\'c, S. Theriault.
\emph{Homotopy theory in toric topology.} Uspekhi Mat. Nauk~\textbf{71} (2016), no.~2, 3--80 (Russian); Russian Math. Surveys~~\textbf{71} (2016), no.~2, 185--251 (English translation).

%+
\bibitem{hardie}
K. A. Hardie. 
\emph{Higher Whitehead products.}
Quart. J. Math. Oxford Ser. (2) ~\textbf{12} (1961), 241--249.

%+
\bibitem{NR}
D. Notbohm, N. Ray. 
\emph{On Davis–Januszkiewicz homotopy types I; formality and
rationalisation.} Alg. Geom. Topol. ~\textbf{5} (2005), 31–51.


%+
\bibitem{PR}
T. Panov, N. Ray.
\emph{Categorical aspects of toric topology}.
Toric Topology, M. Harada et al., eds., Contemp. Math.,~\textbf{460}. Amer. Math. Soc., Providence, RI, (2008), 293–322.

%+
\bibitem{porter}
G. Porter.
\emph{Higher-order Whitehead products.}
Topology ~\textbf{3} (1965), 123--135.


%\bibitem{quill}
%D. Quillen.
%\emph{Rational homotopy theory.}
%Ann. of Math. (2) ~\textbf{90} (1969), no. 2,
%205–295.

%??
%\bibitem{tanre}
%D. Tanre.
%\emph{Homotopie Rationnelle: Modeles de Chen, Quillen, Sullivan.}
%Lecture Notes in Mathematics 1025, Springer (1983).

%+
\bibitem{wil}
F. Williams. 
\emph{Higher Samelson products.}
J. Pure Appl. Algebra ~\textbf{2} (1972), 249–260.


\end{thebibliography}
\end{document}